\documentclass[11pt]{article}

\usepackage{amsmath, amscd, amssymb,amsfonts,amsthm,amscd,graphicx,,psfrag,epsfig}
\usepackage{youngtab}
\usepackage{graphicx}
\usepackage[all]{xy}
\usepackage{array}
\usepackage{amsthm}
\usepackage{stmaryrd}
\usepackage[titletoc,toc]{appendix}
\usepackage{stmaryrd}

\newtheorem{theorem}{Theorem}[section]

\newtheorem{lemma}[theorem]{Lemma}
\newtheorem{proposition}[theorem]{Proposition}
\newtheorem{corollary}[theorem]{Corollary}

\newtheorem{conjecture}[theorem]{Conjecture}
\newtheorem{theorem-definition}[theorem]{Theorem-Definition}
\newtheorem{theorem-construction}[theorem]{Theorem-Construction}
\newtheorem{lemma-definition}[theorem]{Lemma--Definition}
\newtheorem{lemma-construction}[theorem]{Lemma--Construction}

\theoremstyle{definition}
\newtheorem{definition}[theorem]{Definition}

\newtheorem{remark}[theorem]{Remark}
\newtheorem{example}[theorem]{Example}

\newcommand{\und}{\underline}

\newcommand{\OO}{\mathcal{O}}

\newcommand{\C}{\mathbb{C}}

\newcommand{\Z}{\mathbb{Z}}

\newcommand{\old}[1]{}

\newcommand{\Q}{{\mathbb Q}}

\newcommand{\U}{{\rm U}}

\newcommand{\val}{{\rm val}}

\newcommand{\hlra}{\lhook\joinrel\longrightarrow}

\newcommand{\lms}{\longmapsto}
\newcommand{\lra}{\longrightarrow}

\newcommand{\ra}{\rightarrow}
\newcommand{\be}{\begin{equation}}
\newcommand{\ee}{\end{equation}}
\newcommand{\bt}{\begin{theorem}}
\newcommand{\et}{\end{theorem}}
\newcommand{\bd}{\begin{definition}}
\newcommand{\ed}{\end{definition}}
\newcommand{\bp}{\begin{proposition}}
\newcommand{\ep}{\end{proposition}}
\newcommand{\blc}{\begin{lemma-construction}}
\newcommand{\elc}{\end{lemma-construction}}

\newcommand{\bl}{\begin{lemma}}
\newcommand{\el}{\end{lemma}}
\newcommand{\bc}{\begin{corollary}}
\newcommand{\ec}{\end{corollary}}
\newcommand{\bcon}{\begin{conjecture}}
\newcommand{\econ}{\end{conjecture}}
\newcommand{\la}{\label}

\input xy
\xyoption{all}

\begin{document}

\title{ Tensor invariants, Saturation problems,  and Dynkin automorphisms }
\author{Jiuzu Hong, Linhui Shen}
\date{}
\maketitle

\begin{abstract}
Let $G$ be a connected almost simple algebraic group with a Dynkin automorphism $\sigma$.
 Let $G_\sigma$ be the connected  almost simple algebraic group associated to $G$ and $\sigma$. 
 %The weights of $G_\sigma$ are identified with the $\sigma$-invariant weights of $G$. 
 We prove that
% an analogue of twinning character formula in the tensor invariant space setting, i.e. 
the dimension of the tensor invariant space of $G_\sigma$
is equal to the trace of $\sigma$ on the corresponding tensor invariant space of $G$.   
We prove that  if $G$ has the saturation property then so does $G_\sigma$. As a consequence, we show that the spin group ${\rm Spin}(2n+1)$ is of  saturation property with factor $2$, which strengthens the results of Belkale-Kumar \cite{BK} and  Sam \cite{Sam}
in the case of type $B_n$. 
%Our results rely on the work of Goncharov-Shen \cite{GS} on the parametrization of basis in the tensor invariant space via the tropical points of configuration space of decorated flags.
\end{abstract}

\tableofcontents

\section{Introduction}
%\subsection{Background}
Let $G$ be a connected almost simple algebraic group with a Dynkin automorphism $\sigma$. One can associate another almost simple algebraic group $G_\sigma$ (see Section \ref{Notations}).   We investigate the relation between the tensor invariant spaces of $G$ and  $G_\sigma$ in this paper. 

In detail, the dominant weights of $G_\sigma$ are identified with the $\sigma$-invariant dominant weights of $G$. Let ${\und \lambda}=(\lambda_1,\ldots, \lambda_n)$ be a sequence of dominant weights of $G_\sigma$.  Denote by $V_{\lambda_i}$ (respectively $W_{\lambda_i}$) the irreducible representation of $G$ (respectively $G_\sigma$) of highest weight $\lambda_i$. We are interested in the pair of tensor invariant spaces
\be
V_{\und \lambda}^G:=(V_{\lambda_1}\otimes \ldots \otimes V_{\lambda_n})^G, \quad W_{\und \lambda}^{G_\sigma}:=(W_{\lambda_1}\otimes \ldots \otimes W_{\lambda_n})^{G_\sigma}.
\ee

\subsection{Main results}
We present two main results relating $V_{\und \lambda}^G$ and $W_{\und \lambda}^{G_\sigma}$.

\subsubsection{Twining formula}
%Our starting point is the   twining character formula originally due to Jantzen \cite{Jan}. 
 Let $\lambda$ be a dominant weight of $G_\sigma$. %It can be identified with  a $\sigma$-invariant dominant weight of $G$. Let $V_\lambda$ (respectively $W_\lambda$) be the irreducible representation of $G$ (respectively $G_\sigma$) of highest weight $\lambda$.
 The Dynkin automorphism $\sigma$ uniquely determines an action $\sigma$ on the representation $V_\lambda$ of $G$ by keeping the highest weight vectors invariant.
 Let $\mu$ be a weight of $G_\sigma$.
 % such that  it preserves the highest weight vector and such that $\sigma(g\cdot v)=\sigma(g)\cdot \sigma(v)$ for arbitrary $g\in G$ and $v\in V$. Let $\mu$ be a weight of $G_\sigma$.
 The twining character formula asserts that the trace of $\sigma$ on the weight space $V_\lambda(\mu)$ is equal to the dimension of $W_\lambda(\mu)$. %Here $\mu$ is a weight of $G_\sigma$ and hence also a weight of $G$. 
 
 The twining character formula is originally due to Jantzen \cite{Jan}. Since then, there has been many different proofs appearing in the literature 
(e.g. \cite{FSS, N1, N2, KLP, H}).
 One of these approaches  uses natural bases of the representations that are compatible with the action $\sigma$.
 %and sets up a correspondence between the  $\sigma$-invariant elements of the basis in $V_\lambda(\mu)$ and the elements of the basis in  $W_\lambda(\mu)$. 
 It was achieved via canonical basis in \cite{KLP}, and via MV cycles in \cite{H}.  
Due to the works of Lusztig \cite{L3}, Berenstein-Zelevinsky \cite{BZ} and Kamnitzer \cite{Ka1}, canonical basis and  MV cycles can be parametrized by many different but equivalent combinatorial objects,  i.e. Lusztig data, BZ patterns, and MV polytopes. These parametrizations are crucially used in the proofs of \cite{KLP} and \cite{H}.  

%\vskip 2mm
The first result of the present paper provides an analogue of the  twining  formula in the setting of tensor invariant spaces. 
Note that the Dynkin automorphism $\sigma$ determines an action $\sigma$ on $V_{\und \lambda}^G$. %\subsection{Twining character formula}
Our first theorem is as follows.

 \begin{theorem}
 \label{Twining_tensor_Multiplicity} 
 %Let $\sigma$ be a Dynkin automorphism of $G$.  
The trace of $\sigma$ on the space $V_{\und \lambda}^G$ is equal to the dimension of $W_{\und \lambda}^{G_{\sigma}}$:
  \be 
  \la{ttm1}
  {\rm trace}~(\sigma: V_{\und \lambda}^{G}\to V_{\und \lambda}^{G})={\dim}~W_{\und \lambda}^{G_\sigma}.
  \ee
\end{theorem}
 Theorem \ref{Twining_tensor_Multiplicity} is proved in Section \ref{Section_Proof_of_Twinning_Multiplicity}.  We remark here that Theorem \ref{Twining_tensor_Multiplicity} implies similar twining formulas for more general multiplicity spaces.  
 
\subsubsection{Saturation property}
\label{Saturation_Problem_Result}
%\begin{definition} 
%\la{def.sat.pro.res}
A reductive group $G$ is said of saturation property  with  factor $k$ if
\begin{itemize}
 \item for any dominant weights $\lambda_1,\lambda_2,\cdots,\lambda_m$ such that $\sum_{i=1}^m \lambda_i$ is in the root lattice of $G$, if 
$(V_{N\lambda_1}\otimes V_{N\lambda_2}\otimes\cdots \otimes V_{N\lambda_m})^G\not= 0$ for some positive integer $N,$ then  $(V_{k\lambda_1}\otimes V_{k\lambda_2}\otimes\cdots \otimes V_{k \lambda_n})^G\not=
0$.
\end{itemize}

Kapovich-Millson $\cite{KM}$ proved that every almost simple group is of  saturation property but with a wild factor.  
There is a general saturation conjecture asserting that  every simply-laced group is of  saturation factor $1$  (\cite{KM2}). When $G={\rm SL}_n$, it was first proved by Knutson-Tao \cite{KT} using honeycombs. A different proof was due to Derksen-Weyman \cite{DW}. 
When $G={\rm Spin}(8)$, it was proved  by Kapovich-Kumar-Millson \cite{KKM}. It is still open for  simply-laced groups of other types.   For a more thorough survey on saturation problems, see \cite[Section 8]{Ku}.

The second result of this paper shows that the saturation property of $G$ implies the saturation property of $G_\sigma$.
\begin{theorem}
\label{Saturation_Problem_Dynkin}
 If $G$ is of saturation property with factor $k$,  then $G_\sigma$  is of saturation property  with factor $c_{\sigma}k $, where
\begin{equation}
\la{satu.fac.s}
c_{\sigma}=\begin{cases} 
 2 \qquad  &\text{ if  }  G  \text{ is not of type }  A_{2n} \text{  and }  \sigma \text{ is  of order }  2,\\
 3 \qquad  &\text{ if  } \sigma  \text{ is of order 3},\\
4  \qquad  &\text{ if  } \sigma  \text{ is of order 2 and } G  \text{ is of type } A_{2n}.
\end{cases}  
\end{equation}
 \end{theorem}
 
Theorem \ref{Saturation_Problem_Dynkin} is proved in Section \ref{Section_Proof_Saturation}.
 
 \vskip 2mm

Non simply-laced groups are expected to be of saturation factor $2$. 
For such groups not of type $G_2$, if we assume the saturation conjecture of simply-laced groups, then it follows from Theorem \ref{Saturation_Problem_Dynkin}.  
 In particular, the works of Knutson-Tao and Derksen-Weyman imply that
 \begin{corollary}
 \la{spin.cor}
 The spin group ${\rm Spin}(2n+1)$ is of saturation property with factor $2$.
 \end{corollary}
 
 \begin{proof}
 We start with $G={\rm SL}_{2n}$ together with a nontrivial Dynkin automorphism $\sigma$. By Example \ref{SL_Spin} in Section \ref{Notations},  $G_\sigma={\rm Spin}(2n+1)$.
Theorem \ref{Saturation_Problem_Dynkin} implies that the saturation factor for ${\rm Spin}(2n+1)$ is $2$.
 \end{proof}

 The %similar 
 idea that the saturation property of big group implies the saturation property of the intimately related small group was also adopted by Belkale-Kumar  \cite{BK}, in which they showed that Knutson-Tao's theorem implies that the saturation factors of ${\rm SO}(2n+1)$ and ${\rm Sp}(2n)$ are $2$. However, the techniques used by them are very different from ours.

%\vskip 2mm

\subsection{Main methods}

The main methods of this paper are the geometric Satake correspondence (\cite{L, G, MV}) and the work of Goncharov-Shen \cite{GS} on parametrizations of bases of tensor invariant spaces. 
\vskip 2mm

Let $G^\vee$ be the Langlands dual group of $G$. Let ${\cal K}:=\C((t))$ and let ${\cal O}:=\C[[t]]$. We consider the {\it affine Grassmannian} of the Langlands dual group
$$
{\rm Gr}_{G^\vee}:=G^\vee({\cal K}) / G^\vee({\cal O}).
$$
The geometric Satake correspondence provides a connection between the geometry of the affine Grassmannian of $G^\vee$ and the representation theory of $G$. 
As a consequence, the top components
of certain {\it cyclic convolution variety}  of $G^\vee$ provides a basis of the corresponding
tensor invariant space  of $G$ (Lemma \ref{Tensor_Invariant_Cycles}).
Following Fontaine-Kamnitzer-Kuperberg \cite{BKK}, we call it the {\it Satake basis}.   

\vskip 2mm

Another main tool of this paper is the {\it tropical points} introduced by Goncharov-Shen \cite{GS}. The set  of tropical points is obtained by  the {\it tropicalization} of the {\it configuration space of decorated flags} of $G^\vee$.  The construction is briefly recalled in  Sections 3.2-3.5. 
Goncharov-Shen [{\it loc.cit.}] shows that there exists a canonical bijection between the set of tropical points and the Satake basis.
%  The main theorem in \cite{GS} is that, there is a canonical bijection between the set of tropical points and the Satake basis.   
When $G^\vee={\rm PGL}_2$, the tropical points are equivalent to the integral laminations on a polygon (\cite[Section 12]{FG}). The latter give rise to canonical bases of the tensor invariant spaces of ${\rm SL}_2$ (\cite[Section 1.3.1]{GS}).
When $G=G^\vee={\rm GL}_m$, it generalizes  the work of Kamnitzer \cite{Ka2} that hives parametrizes the Satake basis of tensor invariant spaces of ${\rm GL}_m$, in the sense that tropical points encapsulate hives when $G={\rm GL}_m$ (\cite[Section 3.1]{GS}).

We would like to point out that it is very crucial for us to follow the work of Goncharov-Shen. Indeed, to prove Theorem \ref{Twining_tensor_Multiplicity} and Theorem \ref{Saturation_Problem_Dynkin},  eventually we need to work with the tropical points related to almost simple groups of arbitrary type. 
%As a consequence, the cardinality of ${\bf C}_{{\und \lambda}, G^\vee}$ is equal to the dimension of  $V_{{\und \lambda}}^G$. We briefly recall the construction of the tropical invariants in Section \ref{}

\subsection{Strategies}
Let $\und \lambda$ be a tuple of dominant weights of $G_\sigma$. %It is identified with a sequence of $\sigma^\vee$-invariant dominant coweights of $G^\vee$. 
 Denote by $\mathcal{B}_{\und \lambda, G}$ (respectively $\mathcal{B}_{\und \lambda, G_\sigma}$) the Satake basis of $V^G_{\und \lambda}$  (respectively $W^{G_\sigma}_{ \und \lambda}$).   Denote by ${\bf C}_{ \und \lambda, G^\vee}$ (respectively ${\bf C}_{\und \lambda, (G_\sigma)^\vee}$) the set of tropical points that parametrize $\mathcal{B}_{\und \lambda, G}$ (respectively $\mathcal{B}_{\und \lambda, G_\sigma}$).

The Satake basis $\mathcal{B}_{\und \lambda, G}$ has remarkable properties. One of them is that, there exists a Dynkin automorphism $\sigma$ of which the action  on  $V^G_{\und \lambda}$ interchanges the elements in Satake basis $\mathcal{B}_{\und \lambda, G}$ (Proposition \ref{Tensor_invariant_cycle_Dynkin}).  As a consequence, the trace of $\sigma$ on $V_{\und \lambda}^G$ is equal to the number of $\sigma$-fixed elements in $\mathcal{B}_{\und \lambda, G}$.

%Let $\sigma^\vee$ be the Dynkin automorphism of $G^\vee$ induced by the automorphism $\sigma$ of $G$. It gives rise to a bijection between the sets of top components of cyclic convolution varieties of $G^\vee$. Proposition \ref{Tensor_invariant_cycle_Dynkin} implies that the trace of $\sigma$ on $V_{\und \lambda}^G$ is equal to the number of $\sigma^\vee$-stable top components of the cyclic convolution variety of $G^\vee$.

The Dynkin automorphism $\sigma$ of $G$ gives rise to a Dynkin automorphism $\sigma^\vee$ of $G^\vee$. The latter induces an automorphism  $\sigma^\vee$ on the set   ${\bf C}_{ \und \lambda, G^\vee}$ (see Section \ref{section3.5.1}).  Theorem \ref{comm.kappa} and Proposition \ref{Tensor_invariant_cycle_Dynkin} assert that the $\sigma^\vee$-action on  ${\bf C}_{\und \lambda, G^\vee}$  is compatible with the $\sigma$-action on  $\mathcal{B}_{\und \lambda, G}$, i.e., the following  diagram commutes
\begin{equation}
\la{commu.diag.3.10}
\xymatrix{
  {\bf C}_{\und \lambda, G^\vee}  \ar[d]^{\sigma^\vee}   \ar[r]^{\simeq} &    \mathcal{B}_{\und \lambda, G} \ar[d]^{\sigma}  \\
  {\bf C}_{\und \lambda, G^\vee}   \ar[r]^{\simeq }     &     \mathcal{B}_{\und \lambda, G}
}
\end{equation}
Therefore the $\sigma^\vee$-fixed points in ${\bf C}_{\und \lambda, G^\vee}$ are in bijection with the $\sigma$-fixed elements in $\mathcal{B}_{\und \lambda, G}$.

Theorem \ref{technical.thm.1} is one of the main technical results for proving Theorem \ref{Twining_tensor_Multiplicity}. It asserts that the $\sigma^\vee$-fixed  points in ${\bf C}_{\und \lambda, G^\vee}$ are in one-to-one correspondence with the points in ${\bf C}_{\und \lambda, (G_\sigma)^\vee}$, i.e., there exists a canonical bijection
\be 
\la{3.10.2015.l}
({\bf C}_{\und \lambda, G^\vee})^{\sigma^\vee }\simeq  {\bf C}_{\und \lambda, (G_\sigma)^\vee}.
\ee 
It  follows directly from \eqref{commu.diag.3.10} and \eqref{3.10.2015.l} that the $\sigma$-fixed elements in $\mathcal{B}_{\und \lambda, G}$ are in bijection with the elements in $\mathcal{B}_{\und \lambda, G_\sigma}$.   
Therefore Theorem \ref{Twining_tensor_Multiplicity} is proved.

\vskip 2mm
%In fact Theorem \ref{technical.thm.1} is also one of the main ingredients for the proof of Theorem \ref{Saturation_Problem_Dynkin}.    
Another main technical result  is the
 summation map constructed in Section \ref{Average_Tropical_Point}
\be 
\la{3.10.2015.h}
\Sigma:~  {\bf C}_{\und \lambda, G^\vee}\lra  ({\bf C}_{c_\sigma \cdot \und \lambda, G^\vee})^{\sigma^\vee},
\ee
 where $c_\sigma$ is the number appearing in Theorem \ref{Saturation_Problem_Dynkin}.   
 It follows that if the set ${\bf C}_{\und \lambda, G^\vee}$ is nonempty then the set $({\bf C}_{c_\sigma \cdot \und \lambda, G^\vee})^{\sigma^\vee}$ is nonempty (Theorem \ref{technical.thm.2}). 
 %These maps are very crucial in the proof of Theorem \ref{Saturation_Problem_Dynkin}.  
By the bijection \eqref{3.10.2015.l},  the set ${\bf C}_{c_\sigma \cdot \und \lambda, (G_\sigma)^\vee}$ is nonempty. 
In this way we transfer the saturation property of $G$ to the saturation property of $G_\sigma$ (see Section \ref{Section_Proof_Saturation}). 

\subsection{Other applications}
Along our proofs of Theorems \ref{Twining_tensor_Multiplicity}, \ref{Saturation_Problem_Dynkin},  we  get several interesting numerical results of representation theory related to $G$ and $G_\sigma$.  
\begin{proposition}
\label{Numerical_Results}
With the same setting as in Theorem \ref{Twining_tensor_Multiplicity} and Theorem \ref{Saturation_Problem_Dynkin},  we have
\begin{enumerate}
\item   $\dim V_{\und \lambda}^G\geq  \dim W_{\und \lambda}^{G_\sigma}$. 
\item  If $\dim V_{\und \lambda}^G = 1$, then $\dim W_{\und \lambda}^{G_\sigma}=1$.
\item  If $\dim V_{\und \lambda}^G  \not=0 $, then $\dim W_{c_\sigma \cdot \und \lambda}^{G_\sigma}\not= 0$.
\end{enumerate}
\end{proposition}
\begin{proof}
The first and the second results follow from Theorem \ref{technical.thm.1}. The third result follows from  Theorem \ref{technical.thm.2}.
\end{proof}

There is a conjecture by Fulton asserting that for a triple ${\und \lambda}=(\lambda_1,\lambda_2, \lambda_3)$ of dominant weights of ${\rm GL}_n$, if $\dim V_{\und \lambda}^{{\rm GL}_n}=1$, then $\dim V_{N\und \lambda}^{{\rm GL}_n}=1$  for all $N\in {\Bbb N}$.   The conjecture was proved by Knutson-Tao-Woodward \cite{KTW} using honeycomb models.  Combining it with Proposition  \ref{Numerical_Results} (2) ,  we get the following result. 
\begin{proposition}
Let  $\und \lambda=(\lambda_1,\lambda_2,\lambda_3)$ be a triple of $\sigma$-invariant dominant weights of ${\rm SL}_n$.  If $\dim V_{\und \lambda}^{{\rm SL}_{n}}=1$, then for all $N \in {\Bbb N}$ we have 
\be
\begin{cases}
 \dim (W_{N\und \lambda})^{{\rm Spin}(n+1)}=1   \qquad & \text{if $n$ is even} \\
  \dim (W_{N\und \lambda})^{{\rm Sp}(n-1)}=1   \qquad  & \text{if $n$ is odd}
  \end{cases},
\ee
where ${\rm Sp}(n-1)$ is the Symplectic group. 
\end{proposition}

\paragraph{Acknowlegements.} 
We would like to thank  P.\,Belkale, A.\,Goncharov, S.\,Kumar and G.\,Lusztig for their helpful comments and suggestions.  We  would also like to thank the anonymous referee for many helpful suggestions.

%%%%%%%%%%%%%%%%%%%%%
%%%%%%%%%%%%%%%%%%%%%
%%%%%%%%%%%%%%%%%%%%%
%%%%%%%%%%%%%%%%%%%%%
%%%%%%%%  SECTION 2 %%%%%%
%%%%%%%%%%%%%%%%%%%%%
%%%%%%%%%%%%%%%%%%%%%
%%%%%%%%%%%%%%%%%%%%%
%%%%%%%%%%%%%%%%%%%%%

\section{Basics of reductive groups}
 Let $G$ be a  connected almost simple group with a Dynkin automorphism $\sigma$.   
In this section, we introduce two different groups $G_\sigma$ and $G^\sigma$ related to $G$.

\subsection{Dynkin automorphisms of $G$}
Let $G$ be a connected almost simple algebraic group over $\C$. Let $T$ be a
maximal torus in $G$ and let $B$ be a Borel subgroup containing $T$. 
%Let $X$ be the lattice of characters of $T$. 
Denote by $X^\vee$ and $X$ the lattices of cocharacters and characters of $T$.  We associate a root datum $(X^\vee,X,\alpha_i^\vee, \alpha_i,
 i\in I)$ to $(G,B,T)$ together with a perfect pairing
  $$\langle~,~\rangle:X^\vee\times X\to \Z.$$
  Here $I$ is the index set of simple coroots $\{\alpha_i^\vee\}$ and simple roots $\{\alpha_i\}$. We have the Cartan matrix $(a_{ij}):=\big(\langle\alpha_i^\vee,\alpha_j\rangle\big)$.

  \vskip 3mm

A diagram automorphism $\sigma$ of the root datum  $(X^\vee,X,\alpha_i^\vee, \alpha_i,
 i\in I)$  consists of automorphisms of $X^\vee$ and of $X$, and a permutation of $I$ (without confusion, all of them are denoted by $\sigma $) such that  
\begin{enumerate}
\item  $\langle \sigma (\lambda^\vee),\sigma (\mu)\rangle=\langle\lambda^\vee,\mu\rangle$ for any $\lambda^\vee\in X^\vee$ and $\mu\in X$.
\item  $\sigma(\alpha_i)= \alpha_{\sigma(i)}$ and $\sigma(\alpha_i^\vee)=\alpha_{\sigma(i)}^\vee$.
\end{enumerate}
  
  \vskip 3mm

Let $x_i:\C\to G$ and $y_i: \C\to G$ be root subgroups associated
to the simple roots $\alpha_i$ and $-\alpha_i$. 
The datum $(T,B,x_i,y_i;i\in I)$ is called a pinning of $G$ if it gives rise
to a homomorphism $\gamma_i:{\rm SL}_2\to G$ for each $i\in I$ such that 
\be
\gamma_i(\begin{pmatrix}1 & a \\
0 & 1 \\
\end{pmatrix})=x_i(a), 
\quad
 \gamma_i(\begin{pmatrix}1 & 0 \\
a & 1 \\
\end{pmatrix})=y_i(a), 
\quad
\gamma_i(\begin{pmatrix}a & 0 \\
0 & a^{-1} \\
\end{pmatrix})=\alpha^\vee_i(a).
\ee
Let $\sigma$ be an automorphism of $G$ that preserves $B$ and $T$. It
induces a diagram automorphism of the root datum $(X^\vee,X,\alpha_i,\alpha_i^\vee;i\in I)$, which is still denoted by $\sigma$. We call $\sigma$ a {\it Dynkin automorphism} of $G$ if it  preserves a pinning of $G$, i.e.,
$$\sigma\big(x_i(a)\big)=x_{\sigma(i)}(a), \quad
\sigma\big(y_i(a)\big)=y_{\sigma(i)}(a), \quad \sigma\big(\alpha_i^\vee(a)\big)=\alpha_{\sigma(i)}^\vee(a), \hskip 7mm \forall i\in I.$$

By the isomorphism theorem of the theory of reductive groups (e.g.\cite[Section 9]{Sp}), every diagram automorphism  arises from a Dynkin automorphism of $G$.

 \subsection{The associated group $G_\sigma$} 
 \label{Notations}
Every diagram automorphism $\sigma$ of a root datum $(X^\vee,X,\alpha_i^\vee, \alpha_i,i\in I)$ 
gives rise to the following datum:
\begin{enumerate}
\item
Let $X_\sigma$ be the lattice of $\sigma$-fixed elements in $X$. Let $X_\sigma^\vee :={\rm Hom}_\Z(X_\sigma,\Z)$.
\item
Let $I_\sigma$ be the set of orbits
of $\sigma$ on $I$. For each element $\eta\in I_\sigma$, we set 
$$ \alpha_\eta :=
\begin{cases}
\sum_{i\in
\eta}\alpha_i\quad &\mbox{ if  $a_{ij}=0$ for any two elements $i,j$ in $\eta$}\\
2\sum_{i\in
\eta}\alpha_i\quad &\mbox{ if $\eta=\{i,j\}$ and $a_{ij}=-1$. }
\end{cases}
$$
Note that it covers all  possible cases of $\eta$.
\item The embedding of $X_\sigma$ into $X$ induces a natural map $\theta: X^\vee \to X^\vee_\sigma$. Let $\alpha_\eta^\vee:=\theta(\alpha_i^\vee)$ with $i$ in $\eta$. Clearly $\alpha_\eta^\vee$
does not depend on the choice of $i$.
\end{enumerate}
By \cite[p.29]{Jan} , $(X^{\vee}_\sigma,X_\sigma, \alpha_\eta^\vee,\alpha_\eta, \eta\in I_\sigma)$ is a root datum.
It determines a reductive group $G_\sigma$. 
If $G$ is simply-connected, then so is $G_\sigma$. Here is a table of  $G$ and $G_\sigma$ for nontrivial $\sigma$ (\cite[6.4]{Lu2}):
\begin{enumerate}
\item If $G=A_{2n-1}$ and $\sigma$ is of order $2$, then $G_\sigma=B_n$, $n\geq 2$.
\item If $G=A_{2n}$ and $\sigma$ is of order $2$, then $G_\sigma=C_n$, $n\geq 1$.
\item If $G=D_{n}$ and $\sigma$ is of order $2$, then $G_\sigma=C_{n-1}$, $n\geq 4$.
\item If $G=D_4$ and $\sigma$ is of order $3$, then $G_\sigma=G_2$.
\item If $G=E_6$ and $\sigma$ is of order $2$, then $G_\sigma=F_4$.
\end{enumerate}

\begin{example}
\la{SL_Spin}
Let $G={\rm SL}_{2n}$.  We consider the automorphism 
$$\sigma: G\lra G, \quad g\lms \sigma(g):=w\cdot(g^{t})^{-1} \cdot w^{-1},$$
where $g^t$ is the transposition of $g$ and $w=(w_{ij})$ is a  matrix with entries 
$$w_{ij}:=\begin{cases}
(-1)^{j}   \qquad &\text{ if } i+j=2n+1\\
0    \qquad  &\text{ otherwise }
\end{cases}.$$
The automorphism $\sigma$ is a Dynkin automorphism on $G$.   In this case,  $G_\sigma={\rm Spin}(2n+1)$.
\end{example}

%\begin{corollary}
%Given any finite set of dominant weights $\und \lambda=(\lambda_1,\lambda_2,\cdots,\lambda_n)$
%of $G$. Then $\und \lambda$ satisfies saturation property with defect $k_G$,
%where 
% \begin{equation}
% k_G=\begin{cases}
% 2 \qquad \text{ if } G=Sp_{2m}   \\
% 2 \qquad \text{ if } G=Spin_{2m+1}  \\
% 3   \qquad \text{ if } G=G_2 \text{ and } n=3.\\
% 6  \qquad  \text{ if } G=G_2  \text{ for any } n\geq 3 \\
% 72  \qquad  \text{ if } G = F_4 
% \end{cases}
%  \end{equation} 
%\end{corollary}
%\begin{proof}
% If $G=GL_{2n}$, and $\sigma$ is the following automorphism:
% $$\sigma(g):=\omega_{2n}(g^t)^{-1}\omega_{2n}^{-1} ,$$
% where 
% $$\omega_{2n}=\begin{pmatrix}0 & 0 & \ldots & 0 & 1 \\
% 0 & 0 & \ldots & -1 &0 \\
%\vdots & \vdots & \ddots & 0 & 0 \\
%0 & 1 & \ldots & 0 & 0 \\
%-1 & 0 & \ldots & 0 & 0 \\
%\end{pmatrix} $$
%Then $G_\sigma=Spin_{2n+1}$.%%

%If $G=GL_{2n+1}$, and $\sigma$ is defined as follows:
%$$\sigma(g):=\omega_{2n+1}(g^t)^{-1}\omega_{2n+1}^{-1}, $$
%where $$\omega_{2n+1}=\begin{pmatrix}0 & 0 & \ldots & 0 & 1 \\
%0 & 0 & \ldots & -1 &0 \\
%\vdots & \vdots & \ddots & 0 & 0 \\
%0 & -1 & \ldots & 0 & 0 \\
%1 & 0 & \ldots & 0 & 0 \\
%\end{pmatrix}. $$
%Then $G_\sigma=Sp_{2n}$.

%If $G=E_6$ and $\sigma$ is the nontrivial Dynkin automorphism, then $G_\sigma=F_4$.

%If $G=PSO_8$, and $\sigma$ is the nontrivial Dykin automorphism of order% %$3$, then $G_\sigma=G_2$.

%\end{proof}

\subsection{The fixed point group $G^\sigma$}
\label{Section_Pinning}
%Let $G$ be a connected almost simple algebraic group. 
%Let  $(T,B,x_i,y_i,i\in I)$ be a pinning of $G$.  
Let us fix a pinning $(T,B,x_i,y_i; i\in I)$ of $G$.
Let $\sigma$ be a Dynkin automorphism  of $G$ that preserves the pinning.  
Let $G^\sigma$  be the identity component of the $\sigma$-fixed points of $G$. Let  $T^\sigma$ and $B^\sigma$ be the identity components of the $\sigma$-fixed points of $T$ and $B$ respectively.

Recall the set $I_\sigma$ of orbits of $\sigma$ on $I$.
For each orbit $\eta\in I_\sigma$, there are two cases:
\begin{itemize}
\item[1.] If $a_{ij}=0$ for any $i,j\in \eta$, then we set
$$
x_{\eta}(a):=\prod_{i\in \eta} x_i(a),\quad y_{\eta}(a):=\prod_{i\in \eta}y_i(a),\quad \alpha^\vee_{\eta}:=\sum_{i\in \eta}\alpha_i^\vee.
$$
\item[2.] If $\eta=\{i,j\}$ and $a_{ij}=-1$, then we set
$$
x_{\eta}(a):=x_{i}(a)x_j(2a)x_i(a),\quad y_{\eta}(a):=y_{i}(\frac{a}{2})y_j(a)y_{i}(\frac{a}{2}),\quad \alpha^\vee_{\eta}:=2(\alpha_i^\vee+\alpha_j^\vee).
$$
\end{itemize}
Note that the definition of $x_\eta$ and $y_\eta$ does not depend on the ordering
of elements in $\eta$.

\begin{lemma}
\la{pinning.g.sigma.s}
The datum $(T^\sigma,B^\sigma,x_{\eta}, y_{\eta}; \eta\in I_{\sigma})$ gives a pinning of $G^\sigma$.
\end{lemma}
\begin{proof} The first case is clear. The second case is due to a computation of  ${\rm SL}_3$.
\end{proof}

\begin{remark}
\la{remark.3.13.2015s}
 Let $G^\vee$ be the Langlands dual group of $G$. 
 By considering the diagram automorphism $\sigma$ on the dual root datum of $(X^\vee,X,\alpha_i^\vee, \alpha_i;
 i\in I)$, we get an Dynkin automorphism $\sigma^\vee$ of $G^\vee$.
 Let $(G^\vee)^{\sigma^\vee}$ be the  the identity component of the $\sigma^\vee$-fixed points of $G^\vee$. Note that the cocharacters of $(G^\vee)^{\sigma^\vee}$ are identified with the $\sigma^\vee$-invariant cocharacters of $G^\vee$. So $(G^\vee)^{\sigma^\vee}$ is the Langlands dual group of  $G_\sigma$ (see \cite{KLP}). 
\end{remark}

\paragraph{Weyl groups of $G$ and $G^\sigma$.}
Let $s_i~(i\in I)$ be the simple reflections generating the Weyl group $W$ of $G$. Set $\overline{s}_i:= y_i(1)x_i(-1)y_i(1)$. The elements $\overline{s}_i$ satisfy the braid relations. So we can associate to each $w \in W$ its representative $\overline{w}$ in such a way that for any reduced decomposition $w= s_{i_1}\cdots s_{i_k}$ one has $\overline{w}=\overline{s}_{i_1}\cdots \overline{s}_{i_k}.$
Let $w_0$ be the longest element of the Weyl group. Set $s_G:= \overline{w}_0^2$. Note that $s_G$ is a central element in $G$. Moreover $s_G^2=1$.

The Weyl group $W^\sigma$ of $G^\sigma$ can be naturally embedded into $W$ with generators
\be
\la{11.55.3.16.14h}
s_{\eta}=\begin{cases}\prod_{i\in \eta}s_i,\quad &\mbox{if }a_{ij}=0, ~\forall i,j\in \eta;\\
s_is_js_i, &\mbox{if }\eta=\{i,j\},~a_{ij}=-1.
\end{cases}
\ee
The longest element $w_0$ of $W$ coincides with the longest element of $W^\sigma$. We state the following well-known fact for future use.

\bl 
\la{deco.w0}
Each reduced decomposition $s_{\eta_1}\cdots s_{\eta_m}$ of $w_0$ in $W^\sigma$ determines a reduced decomposition of $w_0$ in $W$ with $s_{\eta_i}$ expressed by \eqref{11.55.3.16.14h}, once we fix an ordering of elements in each $\eta$. 
\el
\begin{example} If the pair $(G, G^\sigma)$ is of Cartan-Killing type $(A_4, B_2)$, then 
$$
w_0=s_{\eta_1}s_{\eta_2}s_{\eta_1}s_{\eta_2}=s_{1}s_4\cdot s_2s_3s_2\cdot s_1s_4\cdot s_2s_3s_2.
$$
\end{example}

We set $\hat{s}_\eta:=y_\eta(1)x_\eta(-1)y_\eta(1)$ for $\eta\in I_\sigma$. There is another representative $\overline{s}_\eta$ of $s_{\eta}$ obtained by its decomposition in $W$. A direct calculation shows that $\hat{s}_{\eta}=h_{\eta}\overline{s}_{\eta}$, where $h_\eta=\alpha_i^{\vee}(2)\alpha_j^\vee(2)$ if $\eta=\{i,j\}, a_{ij}=-1$, and $h_\eta=1$ otherwise.
We associate to $w_0$ a representative $\hat{w}_0$ via a reduced decomposition of $w_0$ in $W^\sigma$. Then
\be
\la{12.25.3.16.2014s}
\hat{w}_0=h\cdot \overline{w}_0, \quad \mbox{ where } 
h:=\begin{cases}
1  \quad & \text{ if } G\not = A_{2n}\\
\prod_{k=1}^n(\alpha_k^\vee(2)\alpha_{2n+1-k}^\vee(2))^k   \quad & \text{ if } G=A_{2n} .
\end{cases} 
\ee
Note that $s_{G}=s_{G^\sigma}:=\hat{w}_0^2$.

\section{Configuration space of decorated flags and its tropicalization}
\label{Configuration_Tropicalization_Section}
In this section, let us assume that $G$ is defined over $\Q$.  Let us fix a pinning $(T, B, x_i, y_i; i\in I)$ of $G$.
Let $\sigma$ be a Dynkin automorphism of $G$ that preserves the pinning.

\subsection{Positive spaces and their tropical points}
Below we briefly introduce the category of positive spaces and the  
tropicalization functor.

%\subsubsection{Positive spaces}
\paragraph{Positive spaces.}
A positive rational function on a split algebraic torus ${\cal T}$ is a nonzero rational function on ${\cal T}$ which in a coordinate system, given by a set of characters of ${\cal T}$, can be presented as a ratio of two polynomials with positive integral coefficients. Denote by $\Q_+({\cal T})$ the set of positive functions on $ {\cal T}$.

A positive structure on an irreducible space (i.e. variety / stack) ${\cal Y}$ is a birational map 
$\gamma$ from ${\cal T}$ to ${\cal Y}$. A rational function
$f$ on ${\cal Y}$ is called positive if $f\circ \gamma \in \Q_{+}({\cal T})$.
Denote by $\Q_{+}({\cal Y})$ the set of positive functions on ${\cal Y}$.
Two positive structures on ${\cal Y}$ are equivalent if they determine the
same set $\Q_+({\cal Y})$. Such a pair $({\cal Y}, \Q_{+}({\cal Y}))$ is
called a positive space. 

Let $({\cal Y}, \Q_{+}({\cal Y}))$ and $({\cal Z}, \Q_+({\cal Z}))$ be a pair of positive spaces. A rational map $\phi: {\cal Y}\ra {\cal
Z}$ is called a positive map if $f\circ \phi \in \Q_{+}({\cal Y})$ for all $f \in \Q_{+}({\cal Z})$.

%\subsubsection{Tropicalization.} 
\paragraph{Tropicalization.}
Let $\big({\cal Y},\Q_+({\cal Y})\big)$ be a positive space. A tropical point of ${\cal Y}$ is a map $l: \Q_+({\cal Y})\ra\Z$ such that
$$
\forall f, g \in \Q_+({\cal Y}), \quad l(f+g)=\min\{l(f), l(g)\},  \quad l(fg)=l(f)+l(g).
$$
Denote by ${\cal Y}(\Z^t)$ the set of tropical points of ${\cal Y}$. 
Tautologically, each $f\in \Q_+({\cal Y})$ determines a $\Z$-valued function $f^t$ of ${\cal Y}(\Z^t)$ such that $f^t(l):=l(f)$. 

The following Lemma is an easy exercise.
\bl
Let $\phi: {\cal Y}\ra {\cal Z}$ be a positive map. There exists a unique map $\phi^t: {\cal Y}(\Z^t)\ra {\cal Z}(\Z^t)$, called the tropicalization of $\phi$,  such that $(f\circ \phi)^t=f^t\circ \phi^t$ for all $f\in \Q_+({\cal Z})$. 
\el

The following lemma is standard. It shows that the tropicalization is a functor from the category of positive spaces to the category of the sets of tropical points.
\begin{lemma}
\label{tropical_composition}
Let $\phi:\cal X\to \cal Y$ and $\psi:\cal Y\to \cal Z$ be two positive maps. Then $(\psi\circ \phi)^t=\psi^t\circ \phi^t$.
\end{lemma}
 
 Let $f, g\in {\Bbb Q}_+({\cal Y})$. We say $f<g$ if $g-f$ is still a positive function on ${\cal
Y}$.
\bl \la{com.pos.2.13.slh}
Let $f, g\in {\Bbb Q}_+({\cal Y})$. If there exists a positive integer $N$
such that $f<g< Nf$, then $f^t=g^t$.
\el
\begin{proof} If $h:=g-f\in {\Q}_+({\cal Y})$, then 
$g^t=\min\{h^t, f^t\}\leq f^t.$
Therefore $(Nf)^t\leq g^t\leq f^t$. Note that $(Nf)^t=f^t$. Therefore $g^t=f^t$.
\end{proof}

%\subsubsection{Example.}
\begin{example}  \la{example.lattice.tropical.pts}
Denote by $X_\ast({\cal T})$ and $X^\ast({\cal T})$ the lattices of cocharacters and characters of  a split algebraic torus ${\cal T}$. There is a perfect pairing 
$$\langle~,~\rangle:  X_\ast({\cal T})\times X^\ast({\cal T})\ra \Z.$$ 
Each $f\in \Q_+({\cal T})$ can be presented as
$$
f=\frac{\sum_{\alpha\in X^*({\cal T})}c_{\alpha} X^\alpha}{\sum_{\alpha\in
X^*({\cal T})}d_{\alpha}X^\alpha},\quad c_{\alpha}, d_{\alpha}\in {\Bbb
N}=\{0,1,2,\cdots\}.
$$
Here $X^\alpha$ is the regular function on $\cal T$ associated to $\alpha$ and $c_\alpha, d_\alpha$ are zero for all but finitely many $\alpha$.   
Each cocharacter $l\in X_*({\cal T})$ determines a tropical point of ${\cal T}$ such that
$$
l(f):=\min_{\alpha~|~c_\alpha \neq 0} \langle l,\alpha\rangle - \min_{\alpha~|~d_\alpha \neq 0} \langle l,\alpha\rangle.
$$
It is easy to show that all tropical points of ${\cal T}$  can be defined this way.
Therefore the set ${\cal T}(\Z^t)$ is canonically identified with $X_*({\cal T})$. We treat them as the same set in this paper.
\end{example}

\bl
\la{convex}
If $f\in \Q_+({\cal T})$ is a regular function\footnote{For example, $f=\frac{1+X^3}{1+X}=1-X+X^2$ is such a function on ${\Bbb G}_m$.} on ${\cal T}$, then $f^t$ is convex, i.e.,
$$
f^t(l_1+l_2)\geq f^t(l_1)+f^t(l_2),\quad \forall l_1, l_2\in X_*({\cal T}).
$$
\el
\begin{proof} The function $f$ is a Laurent polynomial on ${\cal T}$:
$$
f=\sum_{\alpha \in X^*({\cal T})}c_{\alpha}X^\alpha,\quad c_\alpha \in \Z.
$$
It is easy to show that $f^t(l)=\min_{\alpha|c_\alpha \neq 0} \langle l,\alpha\rangle$. The convexity follows.
\end{proof}

\bl
\la{linear}
Let $\phi: {\cal T}_1\ra {\cal T}_2$ be a positive map between two split algebraic tori. If $\phi$ is regular, then $\phi^t: X_{*}({\cal T}_1)\ra X_{*}({\cal T}_2)$ is linear. 
\el
\begin{proof}
Let us write the map $\phi$ in coordinates:
$$
\phi: {\cal T}_1\lra {\cal T}_2, \quad x:=(x_1, \ldots, x_n)\lms (\phi_1(x), \ldots, \phi_m(x)).
$$
If $\phi$ is a regular map, then every $\phi_i(x)$ is invertible. Therefore $\phi_i(x)$ must be monomials of $x_1,\ldots, x_n$ with nontrivial coefficients. So its tropicalization is linear. 
\end{proof}

\vskip 3mm

If the space ${\cal Y}$ admits a positive structure defined by a birational map $\gamma: {\cal T}\ra {\cal Y}$, then $\gamma^t$ is a bijection from ${\cal T}(\Z^t)$ to ${\cal Y}(\Z^t)$.
For $l\in {\cal Y}(\Z^t)$, its pre-image $\beta(l):=(\gamma^t)^{-1}(l)$ is called the coordinate of $l$ in ${\cal T}(\Z^t)$. Note that ${\cal T}(\Z^t)=X_*({\cal T})$ is an abelian group, it induces an extra operation $+_{\gamma}$ on ${\cal Y}(\Z^t)$ such that
\be
\la{2.9.3.20.14s}
\beta(l+_{\gamma}l')=\beta(l)+\beta(l'),\quad l,l'\in {\cal Y}(\Z^t).
\ee

\subsection{Lusztig's positive atlas of $U_\ast$}
\la{sec.lus.pos.atlas.sl}
Let $U=[B,B]$ be the maximal unipotent subgroup inside $B$. Let $B^-$ be the Borel subgroup such that $B\cap B^-=T$. Let $U_{\ast}=U\cap B^-w_0 B^-$.

Let $w_0=s_{i_1}s_{i_2}\ldots s_{i_N}$ be a reduced decomposition in $W$. The sequence  ${\bf i}=(i_1,\ldots,i_N)$ is called a reduced word for $w_0$ in $W$.  There is an open embedding
\be
\la{lustig.pos.chart.ss}
\gamma_{\bf i}: {\Bbb G}_m^N \lhook\joinrel\longrightarrow U_{\ast}, ~~~(a_1,\ldots, a_N)\lms x_{i_1}(a_1)\ldots
x_{i_N}(a_N).
\ee
The birational map $\gamma_{\bf i}$ defines a positive structure of $U_\ast$. It is shown in \cite{Lu1} that all the reduced words for $w_0$ give rise to the equivalent positive structures on $U_{\ast}$, which we call Lusztig's positive atlas.

\vskip 3mm

Note that the Dynkin automorphism $\sigma$ preserves $B$ and $B^-$. So it preserves $U_\ast$.
\bl \la{pos.inv.u}
The automorphism $\sigma: U_{\ast}{\lra} U_{\ast}$ is a positive map. 
\el
\begin{proof}
Let ${\bf i}=(i_1,\ldots, i_N)$ be a reduced word for $w_0$. Then $\sigma({\bf
i})=(\sigma(i_1),\ldots, \sigma(i_N))$ is also a reduced word for $w_0$.
For each $u=x_{i_1}(a_1)\ldots x_{i_N}(a_N)\in U_{\ast}$, we have
$$
\sigma(u)=x_{\sigma(i_1)}(a_1)\ldots x_{\sigma(i_N)}(a_N)\in U_{\ast}.
$$
%Thus $\sigma$ is a positive map from $(U_{\ast},\phi_{\bf i}^*(\Q_+({\Bbb G}_m^l)))$ to $(U_\ast, \phi_{\sigma({\bf i})}^*(\Q_+({\Bbb G}_m^l)))$. 
Since the positive structures given by ${\bf i}$ and $\sigma({\bf i})$ are equivalent, the Lemma follows.
\end{proof}

The tropicalization of $\sigma$ is a bijection 
\be
\sigma^t: U_\ast(\Z^t)\stackrel{\sim}{\lra}U_\ast(\Z^t).
\ee 
Denote by $\big(U_\ast(\Z^t)\big)^\sigma$ the set of $\sigma^t$-fixed points. Below we give a characterization of the $\sigma^t$-fixed points.

Let ${\bf j}=(\eta_1, \ldots, \eta_n)$ be a reduced word for $w_0$ in $W^\sigma$. 
It determines a reduced word ${\bf i}=(i_1,i_2,\ldots, i_N)$ for $w_0$ in $W$ (Lemma \ref{deco.w0}).
The tropicalization of \eqref{lustig.pos.chart.ss} is a bijection
$$
\gamma^t_{{\bf i}}: \Z^N\stackrel{=}{\lra}U_{\ast}(\Z^t).
$$
Denote by $(m_1, \ldots, m_N)$ the pre-image of $l\in U_{\ast}(\Z^t)$ in $\Z^N$, which is called the tropical coordinate of $l$ provided by $\gamma_{{\bf i}}$. 

The following lemma is a manifestation of Proposition 3.5 in \cite{H}.
\bl
\la{u.inv.3.18.s}
A tropical point $l$ is $\sigma^t$-invariant if and only if 
$$m_1=m_2=\ldots=m_{r_{\eta_1}}, \quad m_{r_{\eta_1}+1}=m_{r_{\eta_1}+2}=\ldots=m_{r_{\eta_1}+r_{\eta_2}}, \quad \ldots,$$ 
where $r_{\eta}$ is the cardinality of the orbit $\eta$.
\el

\begin{proof}
First we prove the case when $G$ is of type $A_2$ and $\sigma$ is of order 2. In this case, the set $I=\{1,2\}$, and $\sigma(1)=2$, $\sigma(2)=1$. So ${\bf i}=(1,2,1)$ is a reduced word of $w_0$ in $W$.
If $u=x_1(a)x_2(b)x_1(c)$, then
$$
\sigma(u)=x_2(a)x_1(b)x_2(c)=x_1(\frac{bc}{a+c})x_2(a+c)x_1(\frac{ab}{a+c}).
$$ 
Let $(m_1,m_2,m_3)$ be the coordinate of $l$. So the coordinate  of $\sigma^t(l)$ is
\be
\la{3.18.2014.a2}
(m_2+m_3-\min\{m_1,m_3\}, \min\{m_1,m_3\}, m_1+m_2-\min\{m_1,m_3\}).
\ee
Note that $l=\sigma^t(l)$ if and only if $m_1=m_2=m_3$. The Lemma follows.

The general case can be reduced to the above case and the case when $G$ is of type $A_1\times \cdots \times A_1$. The latter case follows by a similar but easier argument.
\end{proof}

%\bd
%\la{sum.of.u.3.19.s}Let us impose an extra operation $+_{\bf i}$ on $U_{\ast}(\Z^t)$, such that if the coordinates of $l_1$ and $l_2$ in ${\bf i}$ are $(m_1,\ldots, m_N)$ and $(m_1',\ldots, m_N')$ respectively, then the coordinates of $l_1+_{\bf i}l_2$ in ${\bf i}$ is $(m_1+m_1', \ldots, m_N+m_N')$.\ed

Let $U^\sigma$ be the identity component of the $\sigma$-fixed points of $U$.
The reduced word ${\bf j}$ of $w_0$ in $W^\sigma$ determines a positive structure
of $U^\sigma_\ast$:
$$
\gamma_{\bf j}: {\Bbb G}_m^n \lhook\joinrel\longrightarrow U_{\ast}^\sigma,
~~~(a_1,\ldots, a_n)\lms x_{\eta_1}(a_1)\ldots
x_{\eta_n}(a_n).
$$

\bl
\la{3.21.3.19.s}
The natural embedding $\imath: U_\ast^\sigma \hookrightarrow U_{\ast}$ is
a positive map.
The tropicalization of $\imath$ identifies the set $U_\ast^\sigma(\Z^t)$
with $(U_\ast(\Z^t))^\sigma$.
\el
\begin{proof}
Recall the construction of the pinning of $G^\sigma$ in Section \ref{Section_Pinning}. It provides an explicit
expression of the map $\imath$  using the coordinates provided by $\gamma_{\bf
j}$ and $\gamma_{\bf i}$. Then the positivity of $\imath$ is clear. Then
second part follows directly from Lemma \ref{u.inv.3.18.s}.
\end{proof}

%\subsubsection{The regular function $\chi$}
\paragraph{The additive Whittaker character $\chi$.}
The pinning of $G$ determines an additive character $\chi$ of $U$ such that  
\be
\chi(x_i(a))=a, \quad  \forall i\in I; \hskip 7mm \chi(u_1u_2)=\chi(u_1)+\chi(u_2), \quad \forall u_1, u_2\in U.
\ee 
The following Lemma is clear.
\bl
\la{21.31.3.16.14s}
The restriction of $\chi$ on $U_\ast$ is a positive function. The function $\chi$ is invariant under the automorphism $\sigma$, i.e., $\chi\circ\sigma =\chi$.
\el
Denote by $\chi_\sigma$ the additive Whittaker character of $U^\sigma$ such that  $\chi_\sigma(x_\eta(a))=a$ for $\eta\in I_\sigma$, and $\chi_\sigma(u_1u_2)=\chi_\sigma(u_1)+\chi_\sigma(u_2)$ for $u_1, u_2\in U^\sigma$. The restriction of $\chi_\sigma$ on $U_\ast^\sigma$ is a positive function.
 
\bl  \la{12.1.3.17.14h}
We have $\chi_\sigma^t=\chi^t \circ \imath^t.$
\el
\begin{proof}
It is easy to check that  $\chi\circ \imath (x_\eta(a))=\kappa_\eta  \chi_\sigma(x_\eta(a))$, where 
$$\kappa_\eta=\begin{cases}
1 \quad \text{ if } \eta=\{i\}.\\
2 \quad \text{ if }  \eta=\{i,j\}, \text{ and } a_{ij}=0 .\\
3  \quad \text{ if } \eta=\{i,j,k\},  \text{ and } a_{ij}=a_{jk}=a_{ik}=0.     \\
4  \quad  \text{ if } \eta=\{i,j\} , \text{ and } a_{ij}=-1.
 \end{cases} $$
Hence in any case, $\chi_\sigma\leq \chi\circ \imath\leq 4\chi_\sigma$. By 
 Lemma \ref{com.pos.2.13.slh}, $\chi_\sigma^t=(\chi\circ \imath)^t$. By Lemma \ref{tropical_composition}, $(\chi\circ \imath)^t=\chi^t\circ\imath^t$. It concludes the proof of our lemma.
\end{proof}

\subsection{Configuration space of decorated flags}
Let ${\cal A}:=G/U$. The elements of ${\cal A}$ are called decorated flags.
The group $G$ acts on ${\cal A}$ on the left. For each $A\in {\cal A}$, the stabilizer $stab_{G}(A)$ is a maximal unipotent subgroup of $G$. 
Let $\pi$ be the natural projection  from ${\cal A}$ to the flag variety ${\cal B}$ such that $\pi(A)$ is the Borel subgroup containing $stab_G(A)$. 
It is easy to show that for each $B\in {\cal B}$, its fiber $\pi^{-1}(B)$ is a $T$-torsor. 

We consider the {\it configuration space} 
\be 
{\rm Conf}_n({\cal A}):=G\backslash {\cal A}^n.
\ee
We say a pair $(B_1, B_2)\in {\cal B}^2$ is {\it generic} if  $B_1\cap B_2$ is a maximal torus of $G$. 
We consider the following open subspace of ${\rm Conf}_n({\cal A})$:
\be
{\rm Conf}_n^\times ({\cal A}):=\{G\backslash (A_1,\ldots, A_n)~|~ (\pi(A_i), \pi(A_{i+1}))\mbox{ is generic for each }  i\in \Z/n\}.
\ee
Below we introduce a positive structure on ${\rm Conf}_{n}^\times({\cal
A})$, following \cite[Section 8]{FG}. We also refer the readers to \cite[Section 6]{GS} for more details. 

\vskip 2mm
We consider the space
\be
{\cal R}:=G\backslash\{(B_1, A, B_2)~|~( \pi(A), B_1), (\pi(A), B_2) \mbox{ are generic} \}\subset G\backslash ({\cal A}\times {\cal B}^2).
\ee
Denote by ${\cal R}_\ast$ the open subspace of ${\cal R}$ with requiring the pair $(B_1, B_2)$ is also generic.
Abusing notation, denote by $U$ the decorated flag corresponding to the  coset of the identity in ${\cal A}$. 

\bl[{\cite[Section 8]{FG}}] 
\la{basic.inv.ed}
There is an isomorphism %\footnote{Note that the definition of ${\bf ed}$ depends the representative $\overline{w}_0$ chosen. Therefore it depends on the pinning of $G$ chosen.} 
${\bf ed}: {\rm Conf}_2^\times ({\cal A})\stackrel{\sim}{\ra}T$ such that
$$
(A_1, A_2)=(U, {\bf ed}(A_1, A_2)\overline{w}_0\cdot U).
$$
There is an isomorphism 
${\bf an}: {\cal R}\stackrel{\sim}{\ra}U$ such that
$$
(B_1,A,B_2)=(B^-, U, {\bf an}(B_1, A, B_2)\cdot B^-).
$$
The restriction of ${\bf an}$ on ${\cal R}_\ast$ is an isomorphism from ${\cal R}_\ast$ to $U_\ast$.
\el

\begin{figure}[h] %  figure placement: here, top, bottom, or page
   \centering
   \includegraphics[width=3in]{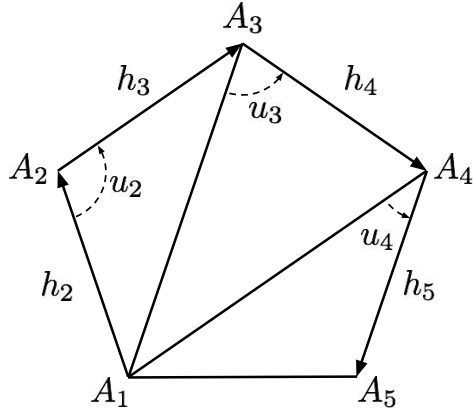} 
   \caption{The invariants assigned to ${\rm Conf}_5^\times({\cal A})$.}
   \la{invaria}
\end{figure}

\blc[{\cite[Section 8]{FG}}] 
\la{le.pos.p2}
There is a natural open embedding\footnote{In fact, the images of $p$ consist of configurations $(A_1,\ldots, A_n)\in {\rm Conf}_{n}^\times ({\cal A})$ such that the pairs $(\pi(A_1), \pi(A_i))$ are also generic for $i=2, \ldots, n$.}
\be
\la{pos.p.2}
p: T^{n-1}\times U_\ast^{n-2} \lhook\joinrel\longrightarrow {\rm Conf}_{n}^\times ({\cal A}),~ (h_2, \ldots, h_n, u_2, \ldots, u_{n-1})\lms (A_1, \ldots, A_{n})
\ee
such that
\begin{itemize}
\item $h_i={\bf ed}(A_{i-1}, A_{i}),~i\in \{2,\ldots,n\};$
\item $u_j={\bf an}(\pi(A_1), A_{j}, \pi(A_{j+1})),~j\in \{2,\ldots, n-1\}.$
\end{itemize}
Let $P_n$ be a convex $n$-gon. Let us assign to each vertex of $P_n$ a decorated flag $A_i$ so that $A_1,\ldots, A_n$ sit clockwise in  the polygon. Then $h_i$, $u_j$ are variables assigned to the edges and angles of $P_n$. See Figure \ref{invaria}.

The positive structure on $U_{*}$ is defined via  Lusztig's atlas. Note that $T$ is a split algebraic group and therefore admits a natural positive structure. So $T^{n-1}\times U_\ast^{n-2}$ admits a positive structure.
From now on, we fix a positive structure on ${\rm Conf}_{n}^\times ({\cal A})$ such that the map $p$ and its inverse $p^{-1}$ are both positive maps. 
\elc

Fock and Goncharov {\cite[Definition 2.5]{FG}} defined the {\it twisted cyclic shift map} 
\be
\la{pos.rot}
r: {\rm Conf}_{n}^\times ({\cal A})\lra {\rm Conf}_{n}^\times ({\cal A}),\quad (A_1,\ldots,
A_{n})\lms (s_{G}\cdot A_n,A_1, \ldots, A_{n-1}).
\ee
They showed that
\bt[{\cite[Corollary 8.1]{FG}}] 
\la{pos.rot.inv}
The twisted cyclic shift map \eqref{pos.rot} is a positive map.
\et

\bc The following map is a regular positive map
\be
{\bf Ed}: {\rm Conf}_{n}^\times ({\cal A})\lra T^n,
\ee
$$
(A_1,\ldots, A_n) \lms \big({\bf ed}(s_G\cdot A_n, A_1), {\bf ed}(A_1, A_2),\ldots, {\bf ed}(A_{n-1}, A_n)\big).
$$
\ec
\begin{proof}
The positivity of the first factor follows from Theorem \ref{pos.rot.inv}. The rest are clear.
\end{proof}

Recall the additive Whittaker character $\chi$ of $U$ in Section \ref{sec.lus.pos.atlas.sl}.
\bd[{\cite[Section 2.1.4]{GS}}] %Recall $\chi$ in Lemma \ref{21.31.3.16.14s}.  
The {\it potential} ${\cal W}$  is a regular function of ${\rm Conf}_n^\times({\cal A})$ such that
$$
{\cal W}(A_1,\ldots, A_n):=\sum_{i\in \Z/n}\chi\big({\bf an}(\pi(A_{i-1}), A_{i}, \pi(A_{i+1}))\big).
$$
\ed
\bc
The potential ${\cal W}$ is a positive function. 
\ec
\begin{proof} 
Note that ${\bf an}(\pi(A_{1}), A_{2}, \pi(A_{3})\big)$ is a part of the map \eqref{pos.p.2}.
By Lemma \ref{21.31.3.16.14s}, $\chi$ is a positive function on $U_\ast$. 
So the summand $\chi\big({\bf an}(\pi(A_{1}), A_{2}, \pi(A_{3})\big)$ is a positive function. The central element $s_G$ is contained in the intersection of Borel subgroups. Therefore
$$
{\bf an}(B_1,s_G \cdot A, B_2 )={\bf an}(B_1, A, B_2).
$$
Using Theorem \ref{pos.rot.inv}, the rest summands are positive functions.
\end{proof}

%\subsection{Configuration space ${\rm Conf}_n^\times({\cal A})$ versus configuration space ${\rm Conf}_n^\times ({\cal A}^\sigma)$}
%\label{bijection_tropical_points}
%In this subsection,  we provide a bijection between the $\sigma^t$-invariant tropical points of $\rm Conf_n^\times ( \cal A )$ and the tropical points of $\rm Conf_n^\times(\cal A^\sigma)$.  This bijection turns out to be compatible with the tropicalization of the potential functions on both spaces.

\subsection{The automorphism $\sigma$ of ${\rm Conf}_n^\times({\cal A})$} 
\la{section3.5.1}
The automorphism $\sigma$ of $G$ preserves $U$. Thus it descends to an automorphism of ${\cal A}$.
Similarly, it descends to an automorphism of ${\cal B}$. Recall the projection $\pi$ from ${\cal A}$ to ${\cal B}$. Clearly $\sigma$ commutes with the projection: $\pi(\sigma (A))=\sigma(\pi(A))$ for $A\in {\cal A}$. 

Abusing notation, we consider the automorphism
\be
\la{pos.ast.1}
\sigma: ~~{\rm Conf}_n^\times ({\cal A})\stackrel{\sim}{\lra} {\rm Conf}_n^\times({\cal A}), ~~~(A_1,\ldots,
A_n)\lra (\sigma(A_1),  \ldots, \sigma(A_n))
\ee

\bl \la{ast.invariants}
The map $\sigma$ commutes with the invariants in Lemma \ref{basic.inv.ed}:\begin{align}
{\bf ed}(\sigma(A_1),\sigma(A_2))&=\sigma({\bf ed}(A_1, A_2)),\la{3.22.2014.h}\\
{\bf an}(\sigma(B_1),\sigma(A), \sigma(B_2))&=\sigma({\bf an}(B_1,A, B_2)).
\end{align}
\el
\begin{proof}
Let $A_1=U$ and let $A_2={\bf ed}(A_1, A_2)\overline{w}_0\cdot U$. Note that $\sigma$ preserves $U$ and $\overline{w}_0$. Therefore  $\sigma(A_1)=U$ and $\sigma(A_2)=\sigma({\bf ed}(A_1, A_2))\overline{w}_0\cdot U$. The first identity follows. The second identity follows by the same argument.
\end{proof}

\bl 
\la{pos.invast}
The automorphism \eqref{pos.ast.1} is a positive map.
\el
\begin{proof}
Recall the birational map $p$ in \eqref{pos.p.2}.
By Lemma \ref{ast.invariants}, we have the isomorphism 
\be
\la{inv.p.3.19.14s}
p^{-1}\circ \sigma \circ p: ~T^{n-1}\times U_\ast^{n-2}\stackrel{\sim}{\lra}T^{n-1}\times U_\ast^{n-2},
\ee
 $$
 (h_2, \ldots, h_n, u_2, \ldots, u_{n-1})\lms (\sigma(h_2), \ldots, \sigma(h_n),\sigma(u_2), \ldots, \sigma(u_{n-1})).
 $$
 By Lemma \ref{pos.inv.u}, it is a positive map. Since $p$ and $p^{-1}$ are
both positive maps, the automorphism $\sigma$ is positive.
\end{proof}

\bl 
\la{wissigmainv}
The automorphism \eqref{pos.ast.1} preserves the potential ${\cal W}$, i.e.,
\be
{\cal W}\big(\sigma(A_1,  \ldots, A_n)\big)={\cal W}(A_1,\ldots,
A_n)
\ee
\el 
\begin{proof}
It follows  from Lemma \ref{ast.invariants} and the fact that $\chi(\sigma(u))=\chi(u)$.
\end{proof}

By Example \ref{example.lattice.tropical.pts}, the set $T(\Z^t)$ is canonically identified with the lattice $X^\vee$ of cocharacters of $T$. 
Let $\underline{\lambda}=(\lambda_1,\ldots, \lambda_n)\in (X^\vee)^n$. 
We set $\sigma(\underline{\lambda}):=(\sigma(\lambda_1),\ldots, \sigma(\lambda_n))$.
\bd
We define the following set of tropical points
\be
{\bf C}_{\underline{\lambda},G}:=\{l\in {\rm Conf}_n^\times({\cal A})(\Z^t)~|~ {\bf Ed}^t(l)=\underline{\lambda},~{\cal W}^t(l)\geq 0\}.
\ee
\ed

\bl 
\la{11.14.3.23.14s}
The tropicalization of \eqref{pos.ast.1} gives rise to a bijection
\be
\la{sigma.14h}
\sigma^t: {\bf C}_{\underline{\lambda},G} \stackrel{\sim}{\lra} {\bf C}_{\sigma(\underline{\lambda}),G}.
\ee
\el
\begin{remark}
For $\underline{\lambda}=\sigma(\underline{\lambda})$, denote by $({\bf C}_{\underline{\lambda},G})^\sigma$ the set of fixed points under \eqref{sigma.14h}.
\end{remark}
\begin{proof}
Let $l\in {\rm Conf}_n^\times ({\cal A})(\Z^t)$. It suffices to prove that
\be
 \la{3.19.eq4}
{\cal W}^t(\sigma^t(l))={\cal W}^t(l),\hskip 7mm
{\bf Ed}^t(\sigma^t(l))=\sigma({\bf Ed}^t(l)).
\ee
The first identity is due to Lemma \ref{wissigmainv}. The second identity is due to \eqref{3.22.2014.h}. %in Lemma \ref{ast.invariants}.
\end{proof}

\subsection{The embedding $\imath$ from ${\rm Conf}_n^\times ({\cal A}^\sigma)$ to  ${\rm Conf}_n^\times ({\cal A})$}
Let ${\cal A}^\sigma:=G^\sigma/ U^\sigma$. By the same construction as \eqref{pos.p.2}, the pinning of $G^\sigma$ in Lemma  \ref{pinning.g.sigma.s} determines  an open embedding
\be
p^{\sigma}: (T^\sigma)^{n-1}\times (U^\sigma_\ast)^{n-2}\lhook\joinrel\longrightarrow {\rm Conf}_n^\times({\cal A}^\sigma).
\ee
It induces a positive structure on the latter space.

There is a natural embedding from ${\cal A}^\sigma$ to ${\cal A}$. It induces a natural embedding
\be
\la{iota}
\iota:~ {\rm Conf}_n^\times({\cal A}^\sigma) \hlra {\rm Conf}_n^\times({\cal A})
\ee
\begin{proposition}
\la{13.11.3.23.14h}
The embedding \eqref{iota} is a positive map. The tropicalization of \eqref{iota} is a bijection from ${\rm Conf}_n^\times({\cal A}^\sigma)(\Z^t)$ to the set of $\sigma^t$-fixed points of ${\rm Conf}_n^\times({\cal A})(\Z^t)$.
\end{proposition}
\begin{proof}
We consider the following composition map
\be
\la{2.12.18.59s}
\jmath: \quad (T^\sigma)^{n-1} \times (U^\sigma_\ast)^{n-2} \stackrel{p^{\sigma}}{\lra}{\rm Conf}_n^\times({\cal A}^\sigma) \stackrel{\iota}{\lra} {\rm Conf}_n^\times({\cal A})\stackrel{p^{-1}}{\lra}
{T}^{n-1}\times{U}^{n-2}.
\ee
Precisely it is given by 
$$
\big(h_2,\ldots, h_n, u_2,\ldots, u_{n-1}\big)\lms \big(h_2h, \ldots, h_{n}h,
u_2, \ldots, u_{n-1}\big),
$$
where $h$ is the element in $T$ described in \eqref{12.25.3.16.2014s}. The element $h$ appears because we use $\hat{w}_0$  instead of $\overline{w}_0$ to define the isomorphism from ${\rm Conf}_2^{\times}({\cal A}^\sigma)$ to $T^\sigma$. Clearly \eqref{2.12.18.59s} is a positive map. Therefore \eqref{iota} is a positive map.

Let us tropicalize the map \eqref{2.12.18.59s}. Note that  $h$ does not contribute to the tropicalization. Therefore we get an injection
$$
\jmath^t: \quad (T^\sigma(\Z^t))^{n-1}\times (U^\sigma_\ast(\Z^t))^{n-2}\lra (T(\Z^t))^{n-1}\times (U_\ast(\Z^t))^{n-2}
$$
$$
(\lambda_2,\ldots, \lambda_{n}, l_2,\ldots, l_{n-1})\lms 
(\lambda_2,\ldots, \lambda_{n}, \imath^t(l_2),\ldots, \imath^t(l_{n-1})).
$$
By Lemma \ref{3.21.3.19.s}, the image of $\jmath^t$ is precisely the set of $\sigma^t$-fixed points. Thus the map $\iota^t$ is a bijection from ${\rm Conf}_n^{\times}({\cal A}^\sigma)(\Z^t)$ to the set of $\sigma^t$-fixed points of ${\rm Conf}_n^{\times}({\cal A}^\sigma)(\Z^t)$.
\end{proof}

Similarly, we have the following positive map / function 
\be
{\bf Ed}_\sigma: {\rm Conf}_n^\times({\cal A}^\sigma)\lra (T^\sigma)^n,\quad \quad {\cal W}_\sigma:{\rm Conf}_n^\times({\cal A}^\sigma)\lra \mathbb{A}^1.
\ee
Let $\underline{\lambda}=(\lambda_1,\ldots, \lambda_n)\in (T^\sigma(\Z^t))^n\subset (T(\Z^t))^n$. Then $\sigma(\underline{\lambda})=\underline{\lambda}$. We set
\be
{\bf C}_{\underline{\lambda}, G^\sigma}:=\{l\in {\rm Conf}_n^\times({\cal A}^\sigma)(\Z^t)~|~ {\bf Ed}_\sigma^t(l)=\underline{\lambda},~{\cal W}_\sigma^t(l)\geq 0\}.
\ee

\bt
\label{technical.thm.1}
The tropicalization of  \eqref{iota} gives rise to a canonical bijection 
\be 
\iota^t: {\bf C}_{\underline{\lambda}, G^\sigma} \stackrel{\sim}{\lra } \big({\bf C}_{\underline{\lambda},G}\big)^\sigma.
\ee
\et

\begin{proof}
It follows from Proposition \ref{13.11.3.23.14h} and the identities
\be
\la{14.50.3.17.14s}
{\cal W}_\sigma^t={\cal W}^t\circ \iota^t,\hskip 7mm
{\bf Ed}_\sigma^t={\bf Ed}^t\circ \iota^t.
\ee
The first identity in \eqref{14.50.3.17.14s} is due to Lemma \ref{12.1.3.17.14h}. The second identity follows similarly.
\end{proof}

\subsection{Summation of tropical points}
\label{Average_Tropical_Point}
Let $\lambda\in X^\vee$. We set
\be
\la{sum.of.trop.lambda}
S(\lambda):=\begin{cases}
\lambda+ \sigma(\lambda); &\text{ if  }  G  \text{ is not of type }  A_{2n}
\text{  and }  \sigma \text{ is  of order }  2,\\
\lambda+ \sigma(\lambda)+ \sigma(\sigma(\lambda)) \qquad  &\text{ if  } \sigma
 \text{ is of order 3},\\
\lambda+\sigma(\lambda)+\sigma\big(\lambda+\sigma(\lambda)\big)  &\mbox{
if } \sigma \text{ is of order 2 and } G  \text{ is of type } A_{2n}.\\
\end{cases}
\ee
It is easy to show that $S(\lambda)$ is $\sigma$-invariant. In particular, if $\lambda$
is $\sigma$-invariant, then $S(\lambda)=c_{\sigma}\lambda$, where $c_{\sigma}$
is described in \eqref{satu.fac.s}.
 
\vskip 3mm

Let us fix a reduced word ${\bf i}$ for $w_0$ in $W$ induced by a reduced word ${\bf j}$ for $w_0$ in $W^\sigma$. 
By \eqref{2.9.3.20.14s}, the the Lusztig
atlas $\gamma_{\bf i}$ determines an operation
on $U_\ast(\Z^t)$,  which we denote by $+_{\bf i}$ for short. Let $l\in \U_\ast(\Z^t)$. We set
\be
\la{sum.of.trop.u}
S_{\bf i}(l):=\begin{cases}
l+_{\bf i} \sigma^t(l); &\text{ if  }  G  \text{ is not of type }  A_{2n}
\text{  and }  \sigma \text{ is  of order }  2,\\
l+_{\bf i} \sigma^t(l)+_{\bf i} \sigma^t\circ\sigma^t(l) \qquad  &\text{
if  } \sigma  \text{ is of order 3},\\
l+_{\bf i}\sigma^t(l)+_{\bf i}\sigma^t\big(l+_{\bf i}\sigma^t(l)\big)  &\mbox{
if } \sigma \text{ is of order 2 and } G  \text{ is of type } A_{2n}.\\
\end{cases}
\ee

\bl
\la{lemma.summ.s}
We have $S_{\bf i}(l)\in \big(U_\ast(\Z^t)\big)^\sigma$.
\el

\begin{proof}
We prove the case when $G$ is of type $A_2$ and $\sigma$ is of order 2. The
other cases follow by a similar but easier argument.

Note that ${\bf i}=(1,2,1)$ is a reduced word of $w_0$ in $W$. 
Let $(m_1, m_2, m_3)$ be the coordinate of $l$ provided by $\gamma_{\bf i}$. The coordinate of $\sigma^t(l)$ is given by \eqref{3.18.2014.a2}. Thus
the coordinate of $l+_{\bf i}\sigma^t(l)$ is $(n,m,n)$, where  
$$n=m_1+m_2+m_3-\min\{m_1,m_3\},\quad \quad  m=m_2+\min\{m_1,m_3\}.
$$
Using \eqref{3.18.2014.a2} again, the coordinate of $\sigma^t\big(l+_{\bf
i}\sigma^t(l)\big)$ is
$(m,n,m)$.
The coordinate of $S_{\bf i}(l)$ is
$$
(m+n, m+n, m+n)=(m_1+2m_2+m_3, m_1+2m_2+m_3, m_1+2m_2+m_3).
$$
By Lemma \ref{u.inv.3.18.s}, $S_{\bf i}(l)$ is $\sigma^t$-invariant.
\end{proof}

\vskip 3mm
%First we give an equivalent definition of the operation $+_{\bf i}$.

Let ${\cal T}:=T^{n-1}\times ({\Bbb G}_m^N)^{n-2}$. 
There is  a chain of open embedding
$$
\Phi_{\bf i}:   \quad {\cal T}=T^{n-1}\times ({\Bbb G}_m^N)^{n-2}\stackrel{{id}\times
\gamma_{\bf i}^{n-2}}{\hlra} T^{n-1}\times (U_{\ast})^{n-2}\stackrel{p}{\hlra}{\rm
Conf}_n^\times({\cal A}).
$$
Its tropicalization is a chain of bijections
 \be \la{12.15.3.23.14s}
 \Phi^t_{\bf i}: \quad {\cal T}(\Z^t)\stackrel{=}{\lra}\big(T(\Z^t)\big)^{n-1}\times
(U_\ast(\Z^t))^{n-2}\stackrel{=}{\lra} {\rm Conf}_n^\times({\cal A})(\Z^t).
\ee
By \eqref{2.9.3.20.14s}, $\Phi^t_{\bf i}$ induces an operation on ${\rm Conf}_n^\times({\cal
A})(\Z^t)$, which is denoted by $+_{\bf i}$ for short.

\bl 
\la{2.56.3.19.s}
Let $l, l'\in {\rm Conf}_n({\cal A})(\Z^t)$.  We have 
\begin{align}
&{\cal W}^t(l+_{\bf i}l')\geq {\cal W}^t(l)+{\cal W}^t(l'), \la{3.19.eq1}\\
&{\bf Ed}^t(l+_{\bf i}l')={\bf Ed}^t(l)+{\bf Ed}^t(l'),\la{3.19.eq2}
\end{align}
\el

\begin{proof}

Note that ${\cal W}\circ \Phi_{\bf i}$ is a regular function of ${\cal T}$.
The inequality \eqref{3.19.eq1} follows from Lemma \ref{convex}.
Note that ${\bf Ed}\circ \Phi_{\bf i}$ is a regular map from the torus ${\cal
T}$ to the torus $T^n$. The identity \eqref{3.19.eq2} follows from
Lemma \ref{linear}.
\end{proof}

\blc
\la{3.00.3.19.s}
Let $l\in{\rm Conf}_n({\cal A})(\Z^t)$. We set
$$
\Sigma_{\bf i}(l):=
\begin{cases}
l+_{\bf i} \sigma^t(l); &\text{ if  }  G  \text{ is not of type }  A_{2n}
\text{  and }  \sigma \text{ is  of order }  2,\\
l+_{\bf i} \sigma^t(l)+_{\bf i} \sigma^t\circ\sigma^t(l) \qquad  &\text{
if  } \sigma  \text{ is of order 3},\\
l+_{\bf i}\sigma^t(l)+_{\bf i}\sigma^t\big(l+\sigma^t(l)\big)  &\mbox{ if
}\sigma  \text{ is of order 2 and } G  \text{ is of type } A_{2n}.\\
\end{cases}
$$
Then $\Sigma_{\bf i}(l)$ is $\sigma^t$-invariant.
\elc

\begin{proof} 
Note that the second bijection of \eqref{12.15.3.23.14s}  is given by
the tropicalization of   \eqref{pos.p.2}
$$
p^t: \big(T(\Z^t)\big)^{n-1}\times \big(U_\ast(\Z^t)\big)^{n-2}\stackrel{=}{\lra}
{\rm Conf}_n({\cal A})(\Z^t).
$$
For $l\in {\rm Conf}_n({\cal A})(\Z^t)$, we consider its preimage
$$\beta(l):=(p^t)^{-1}(l)=(\lambda_2,\ldots, \lambda_{n}, l_2,\ldots, l_{n-1})\in
\big(T(\Z^t)\big)^{n-1}\times \big(U_\ast(\Z^t)\big)^{n-2}.$$
Using \eqref{sum.of.trop.lambda}\eqref{sum.of.trop.u}, we get
$$
\beta\big(\Sigma_{\bf i}(l)\big)=\big(S(\lambda_2),\ldots, S(\lambda_{n}),
S_{\bf i}(l_2),\ldots, S_{\bf i}(l_{n-1})\big).
$$
By Lemma \ref{ast.invariants}, we have
$$
\beta\big(\sigma^t(\Sigma_{\bf i}(l))\big)=\big(\sigma^t(S(\lambda_2)),\ldots,
\sigma^t(S(\lambda_{n})), \sigma^t(S_{\bf i}(l_2)),\ldots, \sigma^t(S_{\bf
i}(l_{n-1}))\big).
$$
By Lemma \ref{lemma.summ.s}, we have $\beta\big(\Sigma_{\bf i}(l)\big)=\beta\big(\sigma^t(\Sigma_{\bf
i}(l))\big)$. Therefore $\Sigma_{\bf i}(l)$ is $\sigma^t$-invariant.
\end{proof}

\bt
\la{technical.thm.2}
If ${\bf C}_{\underline{\lambda},G}$ is nonempty, then  $( {\bf C}_{S(\underline{\lambda}),G})^\sigma$
is nonempty.
\et

\begin{proof}
If ${\bf C}_{{\underline{\lambda}, G}}$ is nonempty, then we pick an element
$l\in {\bf C}_{{\underline{\lambda}, G}}$. It suffices to show that $\Sigma_{\bf i}(l)\in  \big({\bf C}_{S(\underline{\lambda}), G}\big)^\sigma$.
By Lemma \ref{11.14.3.23.14s}
and Lemma \ref{2.56.3.19.s}, we get
$$
{\cal W}^t(\Sigma_{\bf i}(l))\geq c_{\sigma}{\cal W}^t(l)\geq 0,\quad \quad
{\bf Ed}^t(\Sigma_{\bf i}(l))= S({\bf Ed}^t(l))=S(\underline{\lambda}).
$$
So $\Sigma_{\bf i}(l)\in {\bf C}_{S(\underline{\lambda}), G}$. By
Lemma \ref{3.00.3.19.s}, $\Sigma_{\bf i}(l)$ is $\sigma^t$-invariant. 
\end{proof}

\section{Affine Grassmannian and Satake basis}
\label{Satake_Section}
\subsection{Top components of cyclic convolution variety}% and tropical points of configuration space}
Let  $Gr_G:=G({\cal K})/G({\cal O})$ be the affine Grassmannian of $G$.
We consider the action of the maximal torus $T$ on $Gr_G$. The fixed points of $T({\cal O})$ on $Gr_G$ consist of $[\lambda]=t^{\lambda}\cdot G(\OO)$, where $\lambda$ is a coweight of $G$ and $t^\lambda\in T({\cal K})$. 

Let $X^\vee_+$ denote the cone of dominant coweights of $G$. For $L_1,L_2\in Gr_G$, there exists a unique  $\lambda\in X^\vee_+$ such that  $G({\cal K})\cdot (L_1,L_2)=G({\cal K})\cdot([1],[\lambda])$. We write $d(L_1, L_2):=\lambda$, which we call the distance from $L_1$ to $L_2$.

%It gives rise to a distance map
%$$d: Gr\times Gr \lra X^\vee_+,\quad \quad (L_1, L_2)\lms \lambda.$$

 Let $\underline{\lambda}=(\lambda_1,\ldots, \lambda_n)\in (X^\vee_+)^n$. We consider the {\it cyclic convolution variety}
\be
Gr_{G, c(\underline{\lambda})}:=\{(L_1, L_2,\ldots, L_n)~|~ L_n=[1];~d(L_{i-1}, L_{i})=\lambda_i \mbox{ for }i\in \Z/n\}.
\ee
The variety $Gr_{G, c(\underline{\lambda})}$ is of  (complex) dimension
$$
ht(\underline{\lambda}):=\langle \rho, \lambda_1+\lambda_2+\ldots+\lambda_n\rangle,
$$
where $\rho$ is the half sum of positive roots of $G$.
Denote by ${\bf T}_{\underline{\lambda}, G}$ the set of  irreducible components of $Gr_{G, c(\underline{\lambda})}$ of dimension $ht(\underline{\lambda})$.

Note that the Dynkin automorphism $\sigma$ of $G$ preserves ${G}({\cal O})$. So it descends to an automorphism of $Gr_G$. 
Clearly $\sigma$ commutes with the distance map:
\be
d(\sigma(L_1), \sigma(L_2))=\sigma(d(L_1, L_2)),\quad \quad \forall L_1, L_2\in Gr_G.
\ee
Therefore we get a natural bijection
$$
\sigma: Gr_{G, c(\underline{\lambda})}\stackrel{\sim}{\lra} Gr_{G, c(\sigma(\underline{\lambda}))},~~~~(L_1,L_2,\ldots, L_n)\lms (\sigma(L_1),\sigma(L_2),\ldots, \sigma(L_n)).
$$
It induces a bijection on the set of top components 
\be
\la{sigmamapontopcomp}
\sigma:  {\bf T}_{\underline{\lambda},G}\stackrel{\sim}{\lra} {\bf T}_{\sigma(\underline{\lambda}),G}.
\ee

\subsection{Parametrization of top components}
First we briefly recall the {\it constructible functions} in \cite[Section 2.2.5]{GS}. 

Let $R$ be a reductive group over $\C$.
Let ${\cal X}$ be a rational space over $\C$. 
We assume that there is a rational left algebraic action of $R$ on ${\cal X}$. 
Let $\C({\cal X})$ be the function field of ${\cal X}$.
Denote by  by $\circ$ the induced right action of $R$ on $\C({\cal X})$: 
$$
\forall g\in R, ~~\forall F\in \C({\cal X}),\hskip 7mm \big(F\circ g\big)(x):= F(g\cdot x).
$$

 Let ${\cal K}({\cal X})$ be the field of rational functions of ${\cal X}$ with ${\cal K}$-coefficients. 
 The valuation of ${\cal K}^\times$ induces a natural valuation map
 $$
 \val: ~~{\cal K}({\cal X})^\times \lra \Z.
 $$
Here $R({\cal K})$ acts on ${\cal K}({\cal X})$ on the right. Each $F\in {\cal K}({\cal X})^\times$ gives rise
to a $\Z$-valued function
\be
D_F: R({\cal K})\lra \Z, ~~~~g\lms \val (F\circ g).
\ee
The action of the subgroup $R({\cal O})$ preserves the valuation of ${\cal K}({\cal X})^\times$ (\cite[Lemma 2.21]{GS}). So $D_F$ descends to a function from $R({\cal K})/R({\cal
O})$ to $\Z$ which we also denotes by $D_F$.

\vskip 3mm

Assume that there is an automorphism $\tau$ of $R$ and an isomorphism $\tau$ of 
${\cal X}$ such that
\be
\la{theta.equiv.s}
\tau(g\cdot x)=\tau(g)\cdot \tau(x),\quad \quad \forall
g\in R,~\forall x\in {\cal X}. 
\ee
We define a field isomorphism $\tau: \C({\cal X})\stackrel{\sim}{\lra} \C({\cal X})$ such that
\be
\la{3.12.2015.1hl}
\forall F\in \C({\cal X}),\quad \tau(F)(x):=F(\tau^{-1}(x)).
\ee
It is easy to check that 
\be
\forall g\in R, ~~\forall F\in \C({\cal X}), \hskip 7mm
\tau(F\circ g)=\tau(F)\circ \tau (g)
\ee
\bl 
\la{lem.theta.equiv.s}
We have $D_{\tau(F)}(\tau(g))=D_{F}(g)$.
\el
\begin{proof}
Note that $\tau$ preserves the valuation of ${\cal K}({\cal X})^\times$. Therefore
$$
D_{\tau(F)}(\tau(g))=\val\big(\tau(F)\circ \tau(g)\big)=\val\big(\tau(F\circ g)\big)=\val\big(F\circ 
g\big)=D_F(g).
$$
\end{proof}

Now let ${\cal X}:={\cal A}^n$ and let $R:=G^n$. The group $G^n$ acts on ${\cal A}^n$ on the left.
Note that the set $\Q_+\big({\rm Conf}_n^\times({\cal A})\big)$ of positive functions of ${\rm Conf}_n^\times({\cal A})$ is contained in ${\cal K}({\cal A}^n)^\times$.
Each $F\in \Q_+\big({\rm Conf}_n({\cal A})\big)$
induces a function
$$
D_F: R({\cal K})/R({\cal O})=Gr_G^n\lra \Z.
$$
which we call a constructible function on $Gr_G^n$.

\bt[{\cite[Theorems 2.20, 2.23]{GS}}]
\la{3.27.3.23.14h}
There is a canonical bijection 
\be
\kappa: {\bf C}_{\underline{\lambda}, G} \lra {\bf T}_{\underline{\lambda},G},\quad\quad l\lms \kappa(l),
\ee
such that for each $F\in \Q_+\big({\rm Conf}_n({\cal A})\big)$, the generic value of the constructible function $D_F$ on the component $\kappa(l)$ is equal to $F^t(l)$.
\et
\begin{remark}
\la{remark.3.27.3.23.14h}
Let $l, l'\in {\rm Conf}_n({\cal A})(\Z^t)$. 
By the definition of tropical points, we have
$$
l=l' ~~\Longleftrightarrow ~~F^t(l)=F^t(l'),~~ \forall F\in {\Q}_+\big({\rm Conf}_n({\cal A})\big).
$$
Therefore to identify two top components in ${\bf T}_{\underline{\lambda},G}$,
it suffices to show that the generic values of all constructible functions on both components are equal. 
\end{remark}
Recall the bijections \eqref{sigma.14h}\eqref{sigmamapontopcomp}.
\bt
\label{comm.kappa}
The following diagram commutes
\begin{displaymath}
    \xymatrix{
        {\bf C}_{\underline{\lambda},G} \ar[r]^{\kappa} \ar[d]_{\sigma^t} & {\bf T}_{\underline{\lambda},G} \ar[d]^{\sigma} \\
        {\bf C}_{\sigma(\underline{\lambda}),G}   \ar[r]^{\kappa}       & {\bf T}_{\sigma(\underline{\lambda}),G} }.
\end{displaymath}
\et
\begin{proof}
Let $\sigma$ be the automorphism  of $G^n$ such that $\sigma(g_1,\ldots, g_n):=(\sigma(g_1),\ldots, \sigma(g_n))$. 
Let $\sigma$ be the isomorphism of ${\cal A}^n$ such that $\sigma(A_1,\ldots, A_n):=(\sigma(A_1),\ldots, \sigma(A_n))$. Note that $\sigma(g\cdot A)=\sigma(g)\cdot \sigma(A)$. So $\sigma$ agrees with $\tau$ map in \eqref{theta.equiv.s}. 

Let $F\in \Q_{+}\big({\rm Conf}_n^\times({\cal A})\big)$ and let $(L_1,\ldots, L_n)\in Gr_{G, c(\underline{\lambda})}$. 
By Lemma  \ref{lem.theta.equiv.s}, we have
$$
D_F(\sigma(L_1),\ldots,\sigma(L_n))=D_{\sigma^{-1}(F)}(L_1,\ldots, L_n).
$$
Let $l\in  {\bf C}_{\underline{\lambda},G}$. 
Then the generic value of $D_F$ on  $\sigma(\kappa(l))$ is equal to the generic value of $D_{\sigma^{-1}(F)}$ on $\kappa(l)$. By Theorem \ref{3.27.3.23.14h} and \eqref{3.12.2015.1hl}, the latter is  $\big(\sigma^{-1}(F)\big)^t(l)=F^t(\sigma^t(l))$, which is the  generic value of $D_F$ on the component  $\kappa(\sigma^t(l))$. By Remark \ref{remark.3.27.3.23.14h}, $\sigma(\kappa(l))$ and $\kappa(\sigma^t(l))$ are the same component.
\end{proof}

Let ${\und \lambda}= \sigma({\und \lambda})$. Denote by $({\bf T}_{{\und \lambda}, G})^\sigma$ the set of $\sigma$-stable top components of $Gr_{c({\und \lambda})}$. 
\bc
\la{3.27.3.23.14hhh}
There is a natural bijection between $({\bf T}_{{\und \lambda}, G})^\sigma$ and ${\bf T}_{{\und \lambda}, G^\sigma}$.
\ec
\begin{proof}
It follows from the following sequence of bijections
$$
({\bf T}_{{\und \lambda}, G})^\sigma\simeq ({\bf C}_{\underline{\lambda}, G})^\sigma \simeq {\bf C}_{{\und \lambda}, G^\sigma} \simeq {\bf T}_{{\und \lambda}, G^\sigma}.
$$
The first bijection is due to Theorem \ref{comm.kappa}. The second bijection is due to Theorem \ref{technical.thm.1}. The third bijection is due to Theorem \ref{3.27.3.23.14h}.
\end{proof}

\subsection{Geometric Satake correspondence and Satake basis}
\la{Satake_Basis}
%Let $Gr_{G}$ be the affine Grassamnnian of $G$. 
Let ${\rm Perv}_{G(\OO)}(Gr_{G})$ be the category of $G(\OO)$-equivariant perverse sheaves.
Let ${\rm Rep}(G^\vee)$ be the category of finite dimensional representations of $G^\vee$.
The geometric Satake correspondence (e.g. \cite[Therorem 14.1]{MV}) asserts that there is an  equivalence of tensor categories 
$$\mathbb H: {\rm Perv}_{G^\vee(\OO)}(Gr_{G})\simeq \rm{Rep}(G^\vee),$$
where $\mathbb H$ is given by the hypercohomology of perverse sheaves. The tensor category structure on ${\rm Perv}_{G(\OO)}(Gr_{G})$ can be defined via the convolution product (e.g. \cite[Section 4]{MV}).

Let $\und \lambda=(\lambda_1,\lambda_2,\cdots,\lambda_n)$ be a sequence of dominant coweights of $G$.
We define the {\it convolution variety}
\be
Gr_{G, \und \lambda}:=\{(L_1,L_2,\ldots, L_n)~|~ d([1], L_1)=\lambda_1;~d(L_{i-1}, L_{i})=\lambda_{i} \mbox{ for } i=2,\ldots, n\}.
\ee
Denote by $\overline{Gr_{G,\und \lambda}}$ the closure of $Gr_{G,\und
\lambda}$ in $(Gr_{G})^n$. Let $IC_{\und \lambda}$ be the IC sheaf supported on $\overline{Gr_{G,\und \lambda}}$. 
There is a natural projection
\be 
p:Gr_{G}^n\to Gr_{G}, \quad \quad p(L_1,L_2,\cdots,L_n)=L_n.
\ee
The convolution products of perverse sheaves in ${\rm Perv}_{G(\OO)}(Gr_{G})$ are defined such that
\be  IC_{\lambda_1}\ast IC_{\lambda_2}\ast \cdots\ast IC_{\lambda_n}=p_*(IC_{\und \lambda}).
\ee
Note that the cyclic convolution variety $Gr_{G,c(\und \lambda)}$ is the fiber
$$
 Gr_{G,c(\und \lambda)}=p^{-1}([1])\cap Gr_{G, \und \lambda}.
$$
Recall that $\rm ht(\und \lambda)=\langle\rho,\sum_{i=1}^n \lambda_i\rangle$. We have 
$$\dim  Gr_{G,\und
\lambda}=2\rm{ht}(\und \lambda), \quad \quad \dim Gr_{G,c(\und
\lambda)}=\rm ht(\und \lambda).
$$

The following lemma explains why  the set of top components of  $Gr_{G, c(\und \lambda)}$ gives a basis in the tensor invariant space $V^{G^\vee}_{\und \lambda}$. It is well-known to experts. We give a proof here since we can not find any proof in the literature. 
\begin{lemma}
\label{Tensor_Invariant_Cycles}
There is a canonical isomorphism $\alpha: V_{\und \lambda}^{G^\vee}\simeq H_{\rm top}(Gr_{G,c(\und
\lambda)},\C)$, where $H_{\rm top}(Gr_{G,c(\und
\lambda)},\C)$ is the top Borel-Moore homology of $Gr_{G,c(\und
\lambda)}$.  As a consequence, the set of top components of $Gr_{G,c(\und
\lambda)}$ provides a basis of $V^{G^\vee}_{\und \lambda}$.
\end{lemma}
\begin{proof}
By the geometric Satake correspondence, there is a  natural isomorphism
$$V_{\und \lambda}^{G^\vee} \simeq {\rm Hom}_{Gr}(IC_{[1]}, IC_{\lambda_1}*IC_{\lambda_2}*\cdots*IC_{\lambda_n})={\rm Hom}_{Gr_{G}}(IC_{[1]}, p_\ast (IC_{\und \lambda})).$$
We define the embedding $i_{\rm pt}: {\rm pt} \to Gr_G$ by setting  $i_{\rm pt}(\rm pt)=[1]$.
Then $IC_{[1]}=(i_{\rm pt})_*(\C)$. By the adjunction between $(i_{\rm pt})_*$ and $(i_{\rm pt})^!$, we get 
$${\rm Hom}_{Gr_{G}}\big(IC_{[1]},p_\ast (IC_{\und \lambda})\big)\simeq H^0\big((i_{\rm pt})^!(p_* IC_{\und \lambda})\big). $$
Denote by $j$ the locally closed  embedding from $Gr_{G,\und \lambda}$ to $Gr_{G}^n$. 
By \cite[Lemma 2.43]{GS}, there exist natural isomorphisms
\be
\la{13.51.4.12.14h}
j_! (\C_{Gr_{G,\und \lambda}}[{\rm ht}(\und \lambda)])\simeq IC_{\und \lambda}\simeq j_*(\C_{Gr_{G,\und \lambda}}[2\rm{ht}(\und \lambda)]).
\ee
Let $p^\circ$ be the restriction map of $p$ on  $Gr_{G,\und \lambda}$. Let $i$ be the inclusion $Gr_{G,\und \lambda}\hookrightarrow Gr_{G}^n.$
By \eqref{13.51.4.12.14h}, we get  
$$(i_{\rm pt})^!(p_* IC_{\und \lambda})\simeq (i_{\rm pt})^!p_* j_* (\C_{Gr_{G,\und \lambda}}[2{\rm ht}(\und \lambda)])\simeq (i_{\rm pt})^!(p^\circ)_* (\C_{Gr_{G,\und \lambda}}[2\rm{ht}(\und
\lambda)]). $$
By Poincare-Verdier duality, there is a natural isomorphism
$$(i_{\rm pt})^! p^\circ_* (\C_{Gr_{G^,\und \lambda}}[2{\rm{ht}}(\und
\lambda)])\simeq \mathbb{D}_{\rm pt}\big((i_{\rm pt})^* (p^\circ)_! (\C_{Gr_{G,\und \lambda}}[2{\rm{ht}}(\und
\lambda)])\big).$$
Here $\mathbb D_{\rm pt}$ denotes the duality functor on the complex of vector spaces.

By the base change theorem with respect to the following pull-back diagram
\begin{equation}
\xymatrix{
Gr_{G,c(\und \lambda)} \ar[r]^{i} \ar[d]^{} &  Gr_{G,\und \lambda}
\ar[d]^{p^{\circ}}\\
\rm pt \ar[r]^{i_{\rm pt}} &   Gr_{G} 
},\nonumber
\end{equation} 
we get a natural isomorphism
   $$ (i_{\rm pt})^* (p^\circ)_!(\C_{Gr_{G,\und \lambda}}[2\rm{ht}(\und
\lambda)])\simeq R\Gamma_c(Gr_{G,c(\und \lambda)}, \C_{Gr_{G,c(\und \lambda)}}[2\rm{ht}(\lambda)]). $$
Here $R\Gamma_c$ is the derived functor of global section functor with compact support. 
 Thus 
$$H^0((i_{\rm pt})^!p_*(IC_{\und \lambda}))\simeq  H_c^{2\rm ht(\lambda)}(Gr_{G,c(\und \lambda)}, \C)^*.$$
Note that there exists a natural isomorphism 
$$ H_c^{2\rm{ht}(\lambda)}(Gr_{G,c(\und\lambda)}, \C)^*\simeq H_{\rm top}(Gr_{G,c(\und
\lambda)},\C ).  $$
By the basic fact of Borel-Moore homology, the top components of $Gr_{G,c(\und \lambda)}$ provides a basis in $H_{\rm top}(Gr_{G,c(\und
\lambda)},\C )$. 
Hence our lemma follows.
\end{proof}

Let $\sigma$ be a Dynkin automorphism of $G$.  
 Note that $\sigma$ preserves $G({\cal O})$. Thus it descends to an action 
$\sigma$ on  $Gr_{G}$.
By pulling back sheaves, we get an auto-functor $\sigma^*$ of ${\rm Perv}_{G(\OO)}(Gr_{G})$. By the Tannakian formalism, there is an automorphism  $\tilde \sigma$ of $G^\vee$, such that the following diagram commutes
$$\xymatrix{
{\rm Perv}_{G(\OO)}(Gr_{G})\ar[r]^<<<<<{\mathbb H} \ar[d]^{\sigma^*} & {\rm Rep}(G^\vee) \ar[d]^{(\tilde \sigma)^*} \\
  {\rm Perv}_{G(\OO)}(Gr_{G})  \ar[r]^<<<<<{\mathbb H}  & {\rm Rep}(G^\vee)
}, $$
where $(\tilde \sigma)^*$ is the composition functor $(\rho,V)\mapsto (\rho\circ \tilde \sigma, V)$  for any representation $(\rho, V)$ of $G^\vee$.  

The following lemma asserts that  $\tilde \sigma$ is a Dynkin automorphism of $G^\vee$.
\begin{lemma}[{\cite[Theorem 4.2]{H})}]
\label{Tannakian_Dynkin_automorphism}
The automorphism $\tilde \sigma$ on $G^\vee$ is a Dynkin automorphism  arising
from the automorphism $\sigma$ on the root datum $(X^{\vee},X,\alpha_i^\vee,\alpha_i;i\in
I)$.
\end{lemma}

Abusing notation, denote by $\tilde \sigma$ the actions on $V_{\lambda_i}$, $V_{\und \lambda}$  and $V_{\und \lambda}^{G^\vee}$  induced by the automorphism $\tilde \sigma$ on $G^\vee$.

Let $\sigma^\sharp$ be the action on $V_{\und \lambda}^{G^\vee}$ induced by the interchange map on the components of $Gr_{G,\und c(\lambda)}$ via the natural isomorphism 
 $\alpha: V_{\und \lambda}^{G^\vee}\simeq H_{\rm top }(Gr_{G,c(\und \lambda)},\C)$ as in Lemma \ref{Tensor_Invariant_Cycles}.   
\begin{proposition}
\label{Tensor_invariant_cycle_Dynkin}
The actions of $\tilde \sigma$ and $\sigma^\sharp$ on $V_{\und \lambda}^{G^\vee}$ coincide.   
\end{proposition}
\begin{proof}
We consider the natural  isomorphisms $ \phi_i:   \sigma^*IC_{\lambda_i}\simeq IC_{\lambda_i}$ which are  compatible with the interchange action on  cycles (see \cite[Section 4]{H}). 
Applying the hypercohomology $\mathbb H$, we get automorphisms ${\mathbb H}(\phi_i): V_{\lambda_i}\simeq V_{\lambda_i}$.
  Lemma 4.1 in {\it loc.cit.} shows that
${\mathbb H}(\phi_i)$ coincides with the action $\tilde \sigma$ on $V_{\lambda_i}$. 
 
 Recall that the convolution product in ${\rm Perv}_{G(\OO)}(Gr_{G})$ can also be constructed as the fusion product of sheaves via Beilinson-Drinfeld Grassmannian (\cite[Section 5]{MV}). From this point of view, it is easy to see that the isomorphisms $\phi_i$ give rise to an isomorphism 
 $$ \phi: \sigma^*(IC_{\lambda_1}*IC_{\lambda_2}*\cdots *
IC_{\lambda_n})\simeq IC_{\lambda_1}*IC_{\lambda_2}*\cdots *
IC_{\lambda_n}.$$
Applying the functor ${\Bbb H}$, we get
$${\mathbb H}(\phi): V_{\lambda_1}\otimes V_{\lambda_2}\otimes \cdots \otimes
V_{\lambda_n}\simeq V_{\lambda_1}\otimes V_{\lambda_2}\otimes \cdots \otimes
V_{\lambda_n}.$$
By the proof in  \cite[Proposition 6.1]{MV} that $\mathbb H$ is a tensor functor, we  see that ${\mathbb H}(\phi)$ coincides with the diagonal automorphism ${\mathbb H}(\phi_1)\otimes {\mathbb H}(\phi_2)\otimes \cdots\otimes {\mathbb H} (\phi_n)$. Hence ${\mathbb H}(\phi)$ coincides with the the automorphism $\tilde \sigma$ on $V_{\lambda_1}\otimes V_{\lambda_2}\otimes \cdots \otimes
V_{\lambda_n}$.

By the proof of Lemma \ref{Tensor_Invariant_Cycles}, we have 
 $$
 H^0((i_{\rm pt})^! p_*IC_{\und \lambda})\simeq H_{\rm top}(Gr_{G,c(\und \lambda)},\C)\simeq V_{\und \lambda}^{G^\vee}.
 $$
From the counit of the adjunction between $(i_{\rm pt})_*$ and $i_{\rm pt}^!$, there exists a natural morphism 
\begin{equation}
\label{Adjuntion}
\iota: (i_{\rm pt})_*(i_{\rm pt})^! 
p_*IC_{\und \lambda}\to p_*IC_{\und \lambda}.
\end{equation}
Applying ${\mathbb H}$, we  get the natural inclusion $V_{\und \lambda}^{G^\vee}\hookrightarrow
 V_{\und \lambda}$.

Note that $ IC_{\lambda_1}*IC_{\lambda_2}*\cdots
*IC_{\lambda_n}=p_*(IC_{\und \lambda})$. Then  ${\mathbb H}(p_*IC_{\und \lambda})$
is naturally identified with the intersection cohomology of $\overline{Gr_{G,\und
\lambda}}$. There is a unique isomorphism $ \tilde \phi: \sigma^*IC_{\und
\lambda}\simeq IC_{\und \lambda}$,  induced from the interchange action on cycles classes. By natural constructions of $\phi$ and $\tilde \phi$,   
 the following diagram commutes 
\begin{equation}
\xymatrix{
 p_*\sigma^*IC_{\und \lambda}\ar[d]_{\theta} \ar[r]^<<<<<{p_*(\tilde
\phi)} &  p_*IC_{\und \lambda} \\
 \sigma^*p_*IC_{\und \lambda} \ar[ur]_{\phi} 
},
\end{equation}
where $\theta$ is given by the base-change isomorphism.
Note that
$\mathbb H(\theta)$ is the identity map on $\mathbb H(p_*IC_{\und \lambda})$ and $(i_{\rm{pt}})^!(\theta)$ is the identity map on $(i_{\rm pt})^!(p_*IC_{\und \lambda})$.  Therefore $\mathbb H(p_*(\tilde \phi))=\mathbb H(\phi)$ and $(i_{\rm{pt}})^{!}(p_*(\tilde \phi))=(i_{\rm pt})^{!}(\phi)$.   
 By functorialty of the counit 
$(i_{\rm pt})_*(i_{\rm pt})^! 
\to  id$, we have the following commutative diagram:
\begin{equation}
\xymatrix{
(i_{\rm pt})_*(i_{\rm pt})^! ( p_*\sigma^*IC_{\und \lambda}) \ar[d]^{\iota}  \ar[rr]^{(i_{\rm pt})_*(i_{\rm pt})^!p_*(\tilde
\phi)}  &&    (i_{\rm pt})_*(i_{\rm pt})^! (p_*IC_{\und \lambda} )\ar[d]^{\iota} \\
 p_*\sigma^*IC_{\und \lambda} \ar[rr]^<<<<<<<<<<<<{p_*(\tilde
\phi)} & & p_*IC_{\und \lambda} 
}
\end{equation}

Applying the hyper-cohomology $\mathbb H$ to this commutative diagram,  we can see that  the restriction of  
$\mathbb H(\phi)=\mathbb H(p_*\tilde \phi)$ on $V_{\und \lambda}^{G^\vee}$ coincides
with the automorphism 
$$\xymatrix{V_{\und \lambda}^{G^\vee}\simeq H^0\big((i_{pt})^!p_* IC_{\und \lambda}\big)~~~\ar[r]^{(i_{pt})^!(p_* \tilde\phi)}&~~~H^0\big((i_{pt})^!p_*IC_{\und
\lambda}\big)\simeq V_{\und \lambda}^G}. $$
The map $(i_{pt})^!(p_* \tilde\phi)$ interchanges the homology classes given by the top components of $Gr_{G,c(\und \lambda)}$  in $H_{top}(Gr_{G,c(\und \lambda)},\C)$. Hence the proposition follows.
\end{proof}

\section{Proof of main results}
\label{proof_main_results}
\subsection{Proof of Theorem \ref{Twining_tensor_Multiplicity}}
\label{Section_Proof_of_Twinning_Multiplicity}

Let $\sigma$ be the given Dynkin automorphism of $G$. 
It induces an automorphism $\sigma$ of the root datum $(X^\vee,X,\alpha_i^\vee,\alpha_i;i\in I)$. 
Further, we get an associated automorphism $\sigma^\vee$ of the dual datum $(X,X^\vee,\alpha_i,\alpha_i^\vee;i\in I)$.  Abusing notation, denote by $\sigma^\vee$ the Dynkin automorphism of $G^\vee$ arising from the diagram automorphism $\sigma^\vee$.

As explained as in Section \ref{Satake_Basis},  by Tannakian formalism, the Dynkin automorphism $\sigma^\vee$ on $G^\vee$ induces an Dynkin automorphism $\tilde \sigma$ on $G$, which is compatible with the Sakake basis (Proposition \ref{Tensor_invariant_cycle_Dynkin}).

\begin{lemma}
 \label{Dynkin_automorphisms_Conjugate}
 Let $\sigma_1$ and $\sigma_2$ be Dynkin automorphisms  of $G$ that induce the same diagram automorphism of root datum of $G$.  Denote by  $\sigma_1$ and $\sigma_2$  the induced actions on $V_{\und \lambda}^G$ respectively.
We have
$${\rm trace}\ (\sigma_1: V_{\und \lambda}^{G}\to V_{\und \lambda}^{G})={\rm trace}\
(\sigma_2: V_{\und \lambda}^{G}\to V_{\und \lambda}^{G}). $$
\end{lemma}
\begin{proof}
Assume that $\sigma_1$ preserves a pinning $(B,T, x_i,y_i;i\in I)$ and $\sigma_2$
preserves another pinning $(B,T,x'_i,y'_i;i\in I)$. Let $\psi$ be the automorphism
of $G$ such that its restriction on $T$ is an identity map and
$\psi(x_i(a))=x'_i(a)$ , $\psi(y_i(a))=y_i'(a)$. By isomorphism
theorem of reductive groups, $\psi$ is an inner automorphism of $G$. Clearly $\sigma_2=\psi\circ \sigma_1\circ \psi^{-1}$. Note that the
induced actions $\psi$ and $\psi^{-1}$ on $V_{\und \lambda}$ preserve
$V_{\und \lambda}^G$. Hence the lemma follows.

 \end{proof}
 
By  Lemma \ref{Dynkin_automorphisms_Conjugate}, the Dynkin automorphism $\sigma$ of $G$ can be replaced by the automorphism $\tilde \sigma$. 
By Proposition \ref{Tensor_invariant_cycle_Dynkin}, the trace of $\sigma$ on $V_{\und \lambda}^G$ is equal to the number of $\sigma^\vee$-stable top components of $Gr_{G^\vee,c(\und \lambda)}$.
By Lemma \ref{Tensor_Invariant_Cycles}, the dimension of $W_{\und \lambda}^{G_\sigma}$ is equal to the number of the top components of $Gr_{(G_\sigma)^\vee,c(\und \lambda)}$.
Note that $(G_\sigma)^\vee$ is isomorphic to the identity component group of the $\sigma^\vee$-fixed
points in $G^\vee$ (Remark \ref{remark.3.13.2015s}). To summarize, we have the following sequence
\be
{\rm trace}(\sigma: V_{\und \lambda}^G\ra V_{\und \lambda}^G) \stackrel{{\rm Prop.}\ref{Tensor_invariant_cycle_Dynkin}}{=} \# ({\bf T}_{\und \lambda, G^\vee})^{\sigma^\vee}\stackrel{{\rm Cor.} \ref{3.27.3.23.14hhh}}{=} \# {\bf T}_{\und \lambda, (G_{\sigma})^\vee}\stackrel{{\rm Lem.} \ref{Tensor_Invariant_Cycles}}{=} \dim W_{\und \lambda}^{G_\sigma}. \nonumber
\ee
 Theorem \ref{Twining_tensor_Multiplicity}  is proved.

%By Theorem \ref{3.27.3.23.14h} and Theorem \ref{comm.kappa},  the  $\sigma^\vee$-stable top components in $Gr_{c(\und \lambda)}$ is in one-to-one correspondence with  the $(\sigma^\vee)^t$-invariant tropical points in ${\bf C}_{G^\vee,\und \lambda}$. 
%By Thereom \ref{technical.thm.1}, there is a bijection  $({\bf C}_{G^\vee,\und \lambda})^{\sigma^\vee}\simeq
%{\bf C}_{G^\vee_\sigma,\und \lambda}$. It follows that the $\sigma^\vee$-stable top components in $Gr_{G^\vee,c(\lambda)}$ is in one-to-one correspondence with the top components of $Gr_{G^\vee_\sigma,c(\und \lambda)}$.
 %Hence  Theorem \ref{Twining_tensor_Multiplicity} follows.

\subsection{Proof of Theorem \ref{Saturation_Problem_Dynkin}}
\label{Section_Proof_Saturation}
By Theorem \ref{3.27.3.23.14h} and Lemma \ref{Tensor_Invariant_Cycles}, the dimension of the tensor invariant space $V_{\und \lambda}^G$ is equal to the cardinality of the set ${\bf C}_{{\und \lambda}, G^\vee}$. Therefore  the claim that $G$ is of saturation property with factor $k$ is equivalent to that
\begin{itemize}
\item for any sequence $\underline{\lambda}=(\lambda_1,\ldots, \lambda_n)$ of dominant weights of $G$ such that $\sum_{i=1}^n \lambda_i$ is in the root lattice of $G$, if ${\bf C}_{N\underline{\lambda}, G^\vee}$ is nonempty for some positive integer $N$, then ${\bf C}_{k\underline{\lambda}, G^\vee}$ is nonempty.
\end{itemize}

 Now we prove Theorem \ref{Saturation_Problem_Dynkin}. Let $\und \lambda=(\lambda_1,\lambda_2,\cdots,\lambda_n)$ be a sequence of dominant weights of $G_\sigma$ such that $\sum_{i=1}^n \lambda_i$ is in the root lattice of $G_\sigma$. Assume that there is a positive integer $N$ such that 
${\bf C}_{N {\und \lambda}, (G_\sigma)^{\vee}}$  is nonempty. It remains to show that ${\bf C}_{kc_\sigma\cdot\und \lambda, (G_\sigma)^\vee}$ is nonempty.

Note that $(G_\sigma)^\vee$ is the identity component of the  $\sigma^\vee$-fixed points of $G^\vee$. Theorem  \ref{technical.thm.1} implies that 
$({\bf C}_{N\und \lambda,G^\vee})^{\sigma^\vee}$ is nonempty. It follows that  ${\bf C}_{N\und \lambda,G^\vee}$ is nonempty. By the assumption that $G$ is of saturation property with factor $k$, we know that ${\bf C}_{k\und \lambda,G^\vee}$ is nonempty. Theorem \ref{technical.thm.2} implies that $({\bf C}_{kc_\sigma\cdot\und \lambda, G^\vee})^{\sigma^\vee}$ is nonempty. Then again by Theorem \ref{technical.thm.1}, it follows that ${\bf C}_{kc_\sigma\cdot\und \lambda, G^\vee_\sigma}$ is nonempty.
To summarize, we have the following sequence
\begin{align}{\bf C}_{N\und \lambda, (G_\sigma)^\vee}\neq \emptyset & \stackrel{{\rm Thm.}\ref{technical.thm.1}}{\Longleftrightarrow}  ({\bf C}_{N\und \lambda,G^\vee})^{\sigma^\vee}\neq \emptyset~~~{\Longrightarrow} ~~~{\bf C}_{N\und \lambda,G^\vee}\neq \emptyset ~\stackrel{{\rm Assumption}}{\Longrightarrow}~{\bf C}_{k\und \lambda, G^\vee}\neq \emptyset~ \nonumber\\
&\stackrel{{\rm Thm.}\ref{technical.thm.2}}{\Longrightarrow}~({\bf C}_{kc_\sigma\cdot\und \lambda, G^\vee})^{\sigma^\vee}\neq \emptyset~
\stackrel{{\rm Thm.}\ref{technical.thm.1}}{\Longleftrightarrow}~ {\bf C}_{kc_\sigma\cdot\und \lambda, (G_\sigma)^\vee}\neq \emptyset.\nonumber
\end{align}
Theorem \ref{Saturation_Problem_Dynkin} is proved.
%All ingredients of proofs of Theorem \ref{Twining_tensor_Multiplicity} and Theorem \ref{Saturation_Problem_Dynkin} are contained in Section \ref{Configuration_Tropicalization_Section} and \ref{Satake_Section}.

Jiuzu Hong

Department of Mathematics, Yale University, New Haven, CT 06511. 

   \texttt{jiuzu.hong@yale.edu}
\medskip

Linhui Shen

Department of Mathematics, Northwestern University, Evanston, IL 60208. 

\texttt{linhui.shen@northwestern.edu}

\end{document}